\documentclass[oneside,english]{amsart}
\usepackage[LGR,T1]{fontenc}
\usepackage[latin9]{inputenc}
\usepackage{geometry}
\usepackage{float}
\usepackage{textcomp}
\usepackage{mathtools}
\usepackage{amstext}
\usepackage{amsthm}
\usepackage{amssymb}
\usepackage{graphicx}
\usepackage{esint}

\makeatletter

\DeclareRobustCommand{\greektext}{%
  \fontencoding{LGR}\selectfont\def\encodingdefault{LGR}}
\DeclareRobustCommand{\textgreek}[1]{\leavevmode{\greektext #1}}
\ProvideTextCommand{\~}{LGR}[1]{\char126#1}

\numberwithin{equation}{section}
\theoremstyle{plain}
\newtheorem{thm}{\protect\theoremname}[section]
\theoremstyle{definition}
\newtheorem{defn}[thm]{\protect\definitionname}
\theoremstyle{plain}
\newtheorem{lem}[thm]{\protect\lemmaname}
\theoremstyle{plain}
\newtheorem{prop}[thm]{\protect\propositionname}
\theoremstyle{remark}
\newtheorem*{rem*}{\protect\remarkname}
\theoremstyle{plain}
\newtheorem{cor}[thm]{\protect\corollaryname}
\theoremstyle{remark}
\newtheorem{rem}[thm]{\protect\remarkname}
\theoremstyle{plain}
\newtheorem{assumption}[thm]{\protect\assumptionname}
\theoremstyle{plain}
\newtheorem*{prop*}{\protect\propositionname}
\theoremstyle{definition}
\newtheorem{example}[thm]{\protect\examplename}
\theoremstyle{plain}
\newtheorem*{thm*}{\protect\theoremname}
\theoremstyle{remark}
\newtheorem*{acknowledgement*}{\protect\acknowledgementname}

\usepackage{amsaddr}
\makeatother

\usepackage{babel}
\providecommand{\acknowledgementname}{Acknowledgement}
\providecommand{\assumptionname}{Assumption}
\providecommand{\corollaryname}{Corollary}
\providecommand{\definitionname}{Definition}
\providecommand{\examplename}{Example}
\providecommand{\lemmaname}{Lemma}
\providecommand{\propositionname}{Proposition}
\providecommand{\remarkname}{Remark}
\providecommand{\theoremname}{Theorem}

\begin{document}
\title[Embedding Signature-Changing Manifolds]{Embedding Signature-Changing Manifolds: A Braneworld and Kaluza-Klein Perspective}
\author{N. E. Rieger}
\begin{abstract}
We investigate a class of semi-Riemannian manifolds characterized
by smooth metric signature changes with a transverse radical. This
class includes spacetimes relevant to cosmological models such as
the Hartle-Hawking \textquotedblleft no boundary\textquotedblright{}
proposal~\cite{Hartle Hawking - Wave function of the Universe},
where a Riemannian manifold transitions smoothly into a Lorentzian
spacetime without boundaries or singularities. For this class, we
prove the existence of global isometric embeddings into higher-dimensional
pseudo-Euclidean spaces. We then strengthen this result by demonstrating
that a specific type of global isometric embedding, which we term
an $\mathcal{H}$--global embedding, also exists into both Minkowski
space and Misner space. For the canonical $n$--dimensional signature-changing
model, we explicitly construct a full global isometric embedding into
$(n+1)$--dimensional Minkowski and Misner spaces, a significantly
stronger result than an $\mathcal{H}$--global embedding for this
specific case.

This embedding framework provides new geometric tools for studying
signature change and braneworlds through the geometry of submanifolds
embedded in a bulk, thus presenting a mathematically well-defined
approach to these phenomena.
\end{abstract}

\keywords{causality, Kaluza-Klein, singular semi-Riemannian geometry, Lorentzian geometry, signature change, isometric embedding, singular metric, braneworlds, Misner space}

\maketitle
\address{Department of Mathematics, Yale University, USA}

\email{n.rieger@yale.edu}

\section{Introduction and Statement of Results}

In this article, we investigate semi-Riemannian manifolds $(\tilde{M},\tilde{g})$
where the metric smoothly changes signature. Our focus is on scenarios
where an initially Riemannian manifold undergoes a continuous transition,
ultimately evolving into a Lorentzian universe without boundaries
or singularities. In metrics with changing signature, the transition
typically involves either an eigenvalue of the metric passing through
zero, leading to metric degeneracy, or a jump from a positive to a
negative value, causing metric discontinuity~\cite{Ellis - Change of signature in classical relativity}.
We primarily adopt the continuous approach, considering cases with
a transverse radical, where the metric $\tilde{g}$ defines a smooth
$(0,2)$--tensor field that becomes degenerate on a subset $\mathcal{H}\subset\tilde{M}$.
This subset $\mathcal{H}$ represents a smoothly embedded hypersurface
in $\tilde{M}$~\cite{Hasse + Rieger-Transformation,Hasse + Rieger-Loops}.\linebreak{}
~\linebreak{}
This specific class, characterized by a transverse radical, represents
a natural and preferred choice for cosmological models, especially
since cosmological applications typically necessitate a spacelike
surface for signature change. In common parlance, in coordinates adapted
to the radical (referred to as radical-adapted Gauss-like coordinates
in~\cite{Hasse + Rieger-Transformation}; see also Definition~\ref{def:Radical-Adapted-Gauss-Like-Coordinates}
in Section~\ref{subsec:Main-results}), the signature-type change
occurs at a single instant in time. Interestingly, this restriction
is also consistent with the stringent constraints on permissible signature
change possibilities imposed by the brane scenario~\cite{Mars et al - Signature Change on the Brane}.
This family of so-called \textit{transverse type-changing singular
semi-Riemannian manifolds} with a transverse radical,\footnote{A \textit{singular} semi-Riemannian manifold has metric tensor (of
arbitrary signature) which is allowed to become degenerate. Furthermore,
we call the metric $g$ a codimension-{}-$1$ \textit{transverse type-changing}
metric if $d(\textrm{det}([\tilde{g}_{\mu\nu}]))_{q}\neq0$ for any
$q\in\mathcal{H}$ and any local coordinate system around $q$.} with $\dim(\tilde{M})\geq2$, can be shown to be \textit{LC (light-cone)-regular}.
This is a crucial notion, verbatim from~\cite{Bernig}:
\begin{defn}
\label{def:LC-regular}A signature-type changing metric $\tilde{g}$
on a smooth manifold $\tilde{M}$ is said to be \textit{LC (light-cone)-regular}
if $0$ is a regular value of $\tilde{g}\in C^{\infty}(T\tilde{M}\setminus\underline{0})$.

~
\end{defn}

\subsection{Preliminary Results\label{subsec:Preliminary-results}}

With Definition~\ref{def:LC-regular}, we can establish the following:
\begin{lem}
\label{lem:LC-regular}Let $(\tilde{M},\tilde{g})$ be an $m$--dimensional
transverse type-changing manifold with a transverse radical. Then
$(\tilde{M},\tilde{g})$ is LC (light-cone)-regular. 
\end{lem}

Since the 1920s, when Kaluza and Klein first proposed that spacetime
might extend beyond the familiar four dimensions, the concept of extra
dimensions has captivated theoretical physicists. In the Kaluza-Klein
framework, these additional dimensions are compactified - typically
envisioned as circular with a radius on the order of the Planck length,
approximately $l_{Pl}\sim10^{-35}$ meters~\cite{Yang - braneworld model in a massive gravityey-4}.
This compactification approach remains a cornerstone of higher-dimensional
unification theories~\cite{Overduin - Kaluza-klein gravity}. More
recently, inspired by string theory, the braneworld scenario has emerged
as an alternative mechanism for concealing extra dimensions~\cite{Arkani-Divali - The Hierarchy Problem and New Dimensions at a Millimeter. ,Maartens + Koyama - Brane-World Gravity,Randall  - An Alternative to Compactification}.
In the 1980's, Rubakov and Shaposhnikov~\cite{Rubakov + Shaposhnikov - Do We Live Inside a Domain Wall}
proposed that our universe is a four-dimensional brane embedded within
a higher-dimensional spacetime. This idea laid the foundation for
the development of later extra-dimensional models and braneworld scenarios.
In the latter, the universe is modeled as a brane, typically a $1+n$
dimensional hypersurface, embedded in a higher-dimensional bulk governed
by the higher-dimensional Einstein field equations.\footnote{Cases with codimension $1$ are more suitable and are most commonly
studied, whereas traditional analytical methods have proven less effective
for cases with codimension $2$ or higher~\cite{Carter - Essentials of Classical Brane Dynamics}.} In contrast to other higher-dimensional theories, the extra dimensions
can be large or even infinite~\cite{Betounes}. A braneworld can
thus be viewed as a spacetime (locally) embedded within a multidimensional
bulk, where the geometry of the embedding is expected to exhibit quantum
fluctuations along the extra dimensions~\cite{Betounes,Maia Kaluza-Klein,Maia and Monte,Monte. Mathematical Support to Braneworld Theory}.

~\linebreak{}
In modern formulations of Kaluza-Klein theory, a higher-dimensional
manifold $(E,g)$ unifies a lower-dimensional spacetime $(M,\tilde{g})$
and a gauge field by utilizing an isometric embedding $f:M\hookrightarrow E$,
where $M$ is a submanifold of the total space $E$~\cite{Monte. Mathematical Support to Braneworld Theory}.
Under this framework, a $4$--dimensional Lorentzian spacetime can
be seen as isometrically embedded in a $5$--principal fiber bundle
over $M$, with a one-dimensional compact fiber. This geometric framework
suggests a unified origin for gravity and other fundamental forces.

~

Together with the fact that the signature-type changing manifolds
under consideration are light-cone (LC) regular, this leads to the
following proposition, which establishes a geometric framework for
extending the braneworld model: We establish the existence of an isometric
embedding into a higher-dimensional pseudo-Euclidean manifold with
signature $(q,p)$, thereby framing the study of signature change
and braneworlds in terms of the geometry of submanifolds embedded
into a bulk of arbitrary dimension, which is a mathematically sound
and well-defined problem. 
\begin{prop}
\label{prop:pseudo-Euclidean embedding}Let $(\tilde{M},\tilde{g})$
be an $(1+n)$--dimensional transverse type-changing manifold with
a transverse radical. Then, there exists an isometric embedding of
$(\tilde{M},\tilde{g})$ in a $(1+n+d)$--dimensional, non-singular
pseudo-Euclidean manifold \textup{$(M,g)$}.
\end{prop}

This existence result enables the rigorous mathematical study of extended
braneworld scenarios where the $(1+n)$--dimensional brane undergoes
signature change, providing a concrete embedding framework for such
cosmological models. In this context, braneworlds are understood more
broadly, as the higher-dimensional bulk is not assumed to be Lorentzian.
Instead, it is treated as a flat Einstein manifold governed by the
higher-dimensional Einstein field equations.\footnote{Although, in the context of physical general relativity, solutions
to Einstein's field equations must lie on Lorentzian manifolds rather
than general semi-Riemannian manifolds, one can still consider Einstein's
field equations on semi-Riemannian manifolds of other signatures from
a purely mathematical standpoint. Such spaces are known as Einstein
manifolds.} The brane is an isometrically embedded manifold that undergoes a
signature-type change and exhibits quantum fluctuations--manifested
as pseudo-timelike loops, as discussed in~\cite{Hasse + Rieger-Loops}---in
regions sufficiently close to the hypersurface where the signature
transition occurs. Recent studies~\cite{3 + 2 cosmology: Unifying FRW metrics in the bulk}
have investigated scenarios in which a spatially flat FRW metric is
embedded in a bulk with a $(2,3)$ signature (i.e., two timelike dimensions),
yielding a unique bulk solution known as the \textquotedblleft M-metric\textquotedblright .

~

Fundamentally, nothing precludes the existence of perfectly regular
(single) branes that undergo a change in signature, from Riemannian
to Lorentzian, while remaining smooth and regular throughout. Both
the brane and the bulk can be entirely regular, even though the signature
change may appear as a dramatic occurrence from the internal perspective
of the brane~\cite{Mars et al - Signature Change on the Brane}. 

\subsection{Main Results\label{subsec:Main-results}}

These observations and preliminary results pave the way for, and culminate
in, the main findings described below. We first introduce the necessary
notions and definitions.
\begin{defn}
\label{def:Radical-Adapted-Gauss-Like-Coordinates} Let $(\tilde{M},g)$
be an $n$--dimensional singular semi-Riemannian manifold, and let
$\mathcal{H}:=\{q\in M\!\!:g\!\!\mid_{q} is\;degenerate\}$ denote
its surface of signature change. The manifold $(\tilde{M},\tilde{g})$
is defined as a transverse, signature-type changing manifold with
a transverse radical if and only if for every point $q\in\mathcal{H}$,
there exists a neighborhood $U(q)\subset\tilde{M}$ and smooth coordinates
$(t,x^{1},\ldots,x^{n-1})$ on $U(q)$ such that the metric tensor
$\tilde{g}$ takes the canonical form: 

\[
\tilde{g}=-t(dt)^{2}+g_{ij}(t,x^{1},\ldots,x^{n-1})dx^{i}dx^{j},
\]
where the indices $i,j$ range from $1$ to $n-1$. These coordinates
$(t,x^{1},\ldots,x^{n-1})$ are called radical-adapted Gauss-like
coordinates.
\end{defn}

The local radical-adapted Gauss-like coordinates allow for the construction
of a larger, globally consistent neighborhood. This leads to the following
definition.
\begin{defn}
\label{def:H-Global}An $\mathcal{H}$--global neighborhood $U$
of $\mathcal{H}$ is an open set in $\tilde{M}$ defined as the union
of such radical-adapted Gauss-like coordinate neighborhoods for all
points on $\mathcal{H}$: 
\[
U=\bigcup_{q\in\mathcal{H}}U(q).
\]
Within this $\mathcal{H}$--global neighborhood $U$, the function
$\mathfrak{h}(t,\mathbf{\hat{x}}):=t$, with $\mathbf{\hat{x}}=(x^{1},\ldots,x^{n-1})$,
serves as an unique \textit{absolute time function}. This function imposes
a natural time direction on $U$, where an increasing value of $\mathfrak{h}$
corresponds to the future.
\end{defn}

\begin{rem*}
It's important to note that while the function $\mathfrak{h}_{q}(t,\mathbf{\hat{x}}):=t$
for every point $q\in\mathcal{H}$ is defined using the radical-adapted
Gauss-like coordinate $t$ within a specific coordinate patch $U(q)$,
the numerical value of this $t$--coordinate itself can vary across
different patches $U(q')$ for $q'\in\mathcal{H}$. However, the function
value $\mathfrak{h}_{q}(t,\mathbf{\hat{x}})$ for any given point
on $\mathcal{H}$ remains consistent regardless of the specific coordinate
patch used to express it. This means that while the coordinate representation
of  \textquotedblleft time\textquotedblright{} might differ locally,
the physical quantity represented by $\mathfrak{h}_{q}$ is globally
well-defined on $\mathcal{H}$. 
\end{rem*}
\begin{cor}
\label{cor: Time Function Foliation}Let $(\tilde{M},\tilde{g})$
be an $n$--dimensional manifold admitting an $\mathcal{H}$--global
neighborhood $U=\bigcup_{q\in\mathcal{H}}U(q)$. Within $U$, we can
single out a time coordinate defining the real-valued, strictly increasing,
smooth absolute time function $\mathfrak{h}(t,\mathbf{\hat{x}}):=t$,
such that the restriction $(\tilde{M},\tilde{g})\mid_{U}$ admits
a decomposition into spacelike hypersurfaces $\{(U_{\varphi})_{t_{i}}\}$
of dimension $n-1$, given as the level sets $\{(U_{\varphi})_{t_{i}}\}=\mathfrak{h}^{-1}(t_{i})=\{p\in U\mid\mathfrak{h}(p)=t_{i}\}$,
$t_{i}\in\mathbb{R}_{>0}$. This family $\{(U_{\varphi})_{t}\}_{t\in\mathbb{R}_{>0}}$
forms a foliation of \  $\bigcup_{t\in\mathbb{R}_{>0}}(U_{\varphi})_{t}$
into disjoint $(n-1)$--dimensional Riemannian manifolds. For each
slice $\{(U_{\varphi})_{t}\}$, the restriction $(\tilde{g}_{0})_{t_{i}}$
of the metric $\tilde{g}_{0}$ endows the pair $((U_{\varphi})_{t_{i}},(\tilde{g}_{0})_{t_{i}})$
with the structure of a Riemannian manifold.
\end{cor}

\begin{rem}
Consequently, with this absolute time function $\mathfrak{h}$, we
can view the $g_{ij}(t,x^{1},\ldots,x^{n-1})$ from Definition~\ref{def:Radical-Adapted-Gauss-Like-Coordinates}
as a $1$--parameter family of $(n-1)$--dimensional metrics. This
family is a smoothly varying collection of metrics
on the spacelike hypersurfaces, parameterized by $t=\mathfrak{h}(t,\mathbf{\hat{x}})$.
Each metric in this set is uniquely determined by the value of
$t$, and the collection may be regarded as a smooth curve or a \textquotedblleft tube\textquotedblright{}
of metrics in the space of all possible Riemannian metrics. 
\end{rem}

\begin{assumption}
\label{assu:Assumption} Throughout this work, all statements involving
an $\mathcal{H}$--global manifold $(\tilde{M},\tilde{g})$ are understood
to hold \emph{within} a fixed $\mathcal{H}$--global neighborhood
$U=\bigcup_{q\in\mathcal{H}}U(q)\subset\tilde{M}$, as given by Definition~\ref{def:H-Global}.
Unless explicitly stated otherwise, all embeddings $\psi:U\to\mathbb{R}^{1,N}$
into Minkowski space and all compositions $\pi\circ\psi:U\to\mathcal{M}_{\mathrm{Misner}}$
are defined only on this neighborhood $U$. Any coordinates $(t,\mathbf{\hat{x}})$
and associated foliations are likewise understood to be restricted
to $U$.
\end{assumption}

Now we proceed to state our main Theorems and Propositions.
\begin{thm}
\label{thm:H-Global Minkowski} Let $(\tilde{M},\tilde{g})$ be an
$n$--dimensional transverse signature-type changing manifold with
a transverse radical, and fix an $\mathcal{H}$--global neighborhood
$U=\bigcup_{q\in\mathcal{H}}U(q)\subset\tilde{M}$. Then there exists
a sufficiently large integer $N$ and an $\mathcal{H}$--global isometric
embedding $\psi:U\to\mathbb{R}^{1,N-1}$ of $U$ (with the induced
metric $\tilde{g}|_{U}$) into Minkowski spacetime $(\mathbb{R}^{1,N-1},\eta)$.
\end{thm}

To further develop this model, we consider a braneworld scenario with
a single circular extra dimension, achieved through purely spatial
compactification. Given the necessity of compactifying an additional
dimension, we propose that the appropriate higher-dimensional space
for this model is a dimensionally augmented Misner space. This manifold
is constructed by imposing a periodic identification in Minkowski
spacetime via an equivalence relation defined by a discrete group
of hyperbolic rotations, thereby realizing the required compactification.
To prove Theorem~\ref{thm:H-Global Misner} and Proposition~\ref{prop:Toy Model Misner},
we first establish a ``transverse foliation result'', using the isometric embedding $\Phi=(\Phi^1,\text{\dots},\Phi^{N-1})$ provided by the Nash embedding theorem 
for Riemannian manifolds~\cite{Nash}:
\begin{prop}
\label{prop:embedding transverse to the orbital foliation}Let $(\tilde{M},\tilde{g})$
be an $n$--dimensional transverse signature-type changing manifold
with a transverse radical, with coordinates $(t,\mathbf{\hat{x}})=(x^{1},\text{\dots},x^{n\text{\textminus}1})$.\linebreak{}
Let $\psi:U\subseteq\tilde{M}\longrightarrow\mathbb{R}^{1,N-1}$,
\[
\psi(t,\mathbf{\hat{x}}):=\left(-\frac{2}{3}(1+t)^{\frac{3}{2}},\Phi^{1}(t,\mathbf{\hat{x}}),\ldots,\Phi^{N-1}(t,\mathbf{\hat{x}})\right)
\]
be the embedding given in Theorem~\ref{thm:H-Global Minkowski}.
Assume the image of $\psi$ lies in the region 
\[
\mathcal{R}:=\{(\tau,y^{1},\dots,y^{N-1})\in\mathbb{R}^{1,N-1}\;|\;y^{1}-\tau>0\}.
\]
Then the image $\mathrm{Im}(\psi)$ is transverse to the orbital foliation
generated by the one-parameter subgroup of hyperbolic rotations $\Gamma\subset O(1,N-1)$
(acting in the $(\tau,y^{1})$--plane), provided that the spatial
tangent vectors of the embedding are non-zero.
\end{prop}

\begin{thm}
\label{thm:H-Global Misner}Let $(\tilde{M},\tilde{g})$ be an $n$--dimensional
transverse signature-type changing manifold with a transverse radical
and with fixed $\mathcal{H}$--global neighborhood $U=\bigcup_{q\in\mathcal{H}}U(q)\subset\tilde{M}$.
Then there exists an $\mathcal{H}$--global isometric embedding $\pi\circ\psi:U\longrightarrow\mathcal{M}_{\mathrm{Misner}}$
of $(U,\tilde{g}|_{U})$ into Misner space $\mathcal{M}_{\mathrm{Misner}}$.
\end{thm}

Next, we consider a scenario where an $n$--dimensional transverse
signature-type changing canonical model is isometrically embedded
as a hypersurface in an $(n+1)$--dimensional Minkowski spacetime.
This constitutes a significantly stronger result than an $\mathcal{H}$--global
embedding for this specific case. This setup features a single extra
dimension, which aligns with established braneworld frameworks.
\begin{prop}
\label{prop: Toy Model Minkowski embedding} Let $(\mathbb{R}^{n},\tilde{g})$
be the $n$--dimensional signature-type changing manifold with the
metric $\tilde{g}=-t(dt)^{2}+\sum_{i=1}^{n-1}(dx^{i})^{2}$. There
exists a global isometric embedding $f:\mathbb{R}^{n}\longrightarrow\mathbb{R}^{1,n}$
into Minkowski space $(\mathbb{R}^{1,n},\eta)$.
\end{prop}

Considering Misner space, we obtain an effective reduction from an
$(n+1)$--dimensional theory (describing the bulk spacetime) to an
$n$--dimensional theory (effective on the brane), accompanied by
additional gauge fields. When $n=4$, this  aligns with the Kaluza-Klein
framework~\cite{Monte. Mathematical Support to Braneworld Theory},
which posits an inherently higher-dimensional universe with compactified
extra dimensions. 
\begin{prop}
\label{prop:Toy Model Misner}Let $(\mathbb{R}^{n},\tilde{g})$ be
the $n$--dimensional signature-type changing toy model with the
metric $\tilde{g}=-t(dt)^{2}+\sum_{i=1}^{n-1}(dx^{i})^{2}$. Then
there exists a global isometric embedding of the manifold $(\mathbb{R}^{n},\tilde{g})$
into $(n+1)$-dimensional Misner space $\mathcal{M}_{\text{Misner}}$.
\end{prop}

Misner space, a well-known solution to Einstein's equations, has a
Killing symmetry that shifts from timelike to spacelike across a compact
Cauchy horizon. Our approach enables the construction of an extended
Kaluza-Klein theory wherein the effective base space exhibits a change
in signature type, even though the higher-dimensional Misner metric
retains its Lorentzian signature everywhere. This distinct scenario
thus realizes what we call a \textquotedblleft signature change without
signature change\textquotedblright .\footnote{This situation may be described as a case of \textquotedblleft signature change without signature change\textquotedblright, in analogy with John Wheeler's notions of \textquotedblleft mass without mass\textquotedblright\ and \textquotedblleft charge without charge\textquotedblright\ in describing black holes. Here, the signature change is an effective property of the lower-dimensional manifold, rather than a fundamental property of the higher-dimensional space.}

\section{\textbf{Signature-type Changing Braneworlds\label{sec:Signature-type-Changing-Braneworlds}}}

A singular semi-Riemannian manifold with signature-type change is
a mathematical structure in differential geometry where the metric
tensor (of arbitrary signature) becomes degenerate and undergoes a
transition in its signature. Initially, the manifold has a Riemannian
metric, but it smoothly transitions to a semi-Riemannian metric, which
includes both positive and negative eigenvalues. These manifolds are
termed \textquotedblleft singular\textquotedblright{} because the
signature change involves regions where the metric is non-invertible
or exhibits divergent behavior, resulting in mild singularities in
the classical sense~\cite{Ellis - Classical change of signature}.
Such manifolds are used to model scenarios in which the nature of
spacetime changes, such as cosmological models that transition from
Riemannian to Lorentzian metrics. For a detailed discussion on signature-type
changing manifolds, we refer the reader to~\cite{Ellis - Classical change of signature,Hasse + Rieger-Transformation,Hasse + Rieger-Loops}. 
\begin{defn}
\label{def:transverse_type_changing_general} A smooth manifold $\tilde{M}$
equipped with a smooth $(0,2)$--tensor field $\tilde{g}$ is called
a \emph{transverse type-changing semi-Riemannian manifold} if $\tilde{g}$
becomes degenerate on a smoothly embedded hypersurface $\mathcal{H}\subset\tilde{M}$,
and for every point $q\in\mathcal{H}$, $d(\textrm{det}([\tilde{g}_{\mu\nu}]))_{q}\neq0$.
The hypersurface $\mathcal{H}$ is the surface of signature change.
\end{defn}

In cosmological applications, a spacelike surface of signature change
is the natural and preferred choice. Thus, our focus is specifically
directed toward the concept of a transverse radical, which ensures
such a spacelike hypersurface of signature change. For a detailed
discussion of the causal character of these hypersurfaces, we refer
the reader to~\cite{Rieger Diss}. For this class of manifolds,
we first establish the following general existence result, which provides
a foundational geometric framework for developing an extended braneworld
model that incorporates signature change. Let's restate the Proposition~\ref{prop:pseudo-Euclidean embedding}
for convenience.
\begin{prop*}
Let $(\tilde{M},\tilde{g})$ be an $(1+n)$--dimensional transverse
type-changing manifold with a transverse radical. Then, there exists
an isometric embedding of $(\tilde{M},\tilde{g})$ in a $(1+n+d)$--dimensional,
non-singular pseudo-Euclidean manifold \textup{$(M,g)$}.
\end{prop*}
To prove Proposition~\ref{prop:pseudo-Euclidean embedding}, we introduce
the following auxiliary result.
\begin{lem}
\label{lem:Lightcone regular}Let $(\tilde{M},\tilde{g})$ be an $m$--dimensional
transverse signature type changing manifold with a transverse radical.
Then $(\tilde{M},\tilde{g})$ is LC (light-cone)-regular.
\end{lem}

\begin{proof}
Recall that, for a $(0,2)$--tensor field $\tilde{g}$, we define
the associated quadratic form as the function, $Q_{\tilde{g}}:T\tilde{M}\longrightarrow\mathbb{R}$,
given by $Q_{\tilde{g}}(p,v)=\tilde{g}_{p}(v,v)$ for all $(p,v)\in T\tilde{M}$. 

~\linebreak{}
First, by Definition~\cite{Bernig}, a metric $\tilde{g}$ of changing
signature on a smooth manifold $\tilde{M}$ is LC (light cone)-regular
if $0$ is a regular value of $\tilde{g}\in C^{\infty}(T\tilde{M}\setminus\underline{0})$.\footnote{Here the image of the zero section in $T\tilde{M}$ is denoted by
$\underline{0}$.} In other words, consider the smooth map $G:=Q_{\tilde{g}}\mid_{T\tilde{M}\setminus\{0\}}:T\tilde{M}\setminus\{0\}\longrightarrow\mathbb{R}$,
given by $(p,v)\longmapsto\tilde{g}_{p}(v,v)$. The metric $\tilde{g}$
is LC-regular if $0$ is a regular value of $G$, meaning that for
all $(p,v)$ satisfying $G(p,v)=\tilde{g}_{p}(v,v)=0$, the differential
$D_{(p,v)}G$ is nonzero (i.e., surjective).

\textbf{~}\\ We aim to prove the following by contradiction: If $(\tilde{M},\tilde{g})$
is not LC-regular, then $(\tilde{M},\tilde{g})$ is not transverse
type-changing, which is equivalent to stating that if $\bar{p},\bar{v}\neq0$
are critical points for $G$ , then $\bar{p}$ is a critical point
for $\det([\tilde{g}])$ and $\bar{p}\in\mathcal{H}$. Specifically,
this means: If $(\bar{p},\bar{v})$ is a critical point for $G$
, then $\bar{v}$ satisfies the null condition 
\[
Q_{\tilde{g}}\mid_{T\tilde{M}\setminus\{0\}}(\bar{p},\bar{v})=\tilde{g}_{ij}(\bar{p})\bar{v}^{j}=0.
\]
First, we consider the derivative of the quadratic form $Q_{\tilde{g}}\mid_{T\tilde{M}\setminus\{0\}}$
associated with the metric along $\bar{v}$ in the fiber directions.
Let $\tilde{g}(\bar{p})(\bar{v},\bar{v})$ be the metric applied to
a vector $\bar{v}$ at a point $\bar{p}$. Its derivative with respect
to a tangent vector $\bar{v}$ satisfies:

\[
D_{\bar{v}}G(\bar{p},\bar{v})=\frac{\partial}{\partial v_{i}}(\tilde{g}_{\bar{p}}(\bar{v},\bar{v}))=2\tilde{g}_{ij}(\bar{p})\bar{v}^{j}.
\]
This derivative is zero if and only if $\bar{v}$ is an element of
the kernel of the metric,i.e., $\bar{v}\in\text{ker}(\tilde{g}(\bar{p}))\setminus\{0\}$,
which is the radical of the metric at point $p$, denoted $\text{Rad}_{\bar{p}}\setminus\{0\}$.
This implies that $\bar{p}\in\mathcal{H}$, since the kernel of the
matrix representation of $\tilde{g}(\bar{p})$ contains a non-trivial
vector. Consequently, the determinant of the metric vanishes at this
point, i.e., $\det([\tilde{g}])(\bar{p})=0$.

\textbf{~}\\ If $G(\bar{p},\bar{v})=\tilde{g}_{ij}(\bar{p})dx^{i}(\bar{v})dx^{j}(\bar{v})=\tilde{g}_{ij}(\bar{p})\bar{v}^{i}\bar{v}^{j}=0$,
and the horizontal differential of $G$ with respect to the base point
$p$, evaluated at $(\bar{p},\bar{v})$, satisfies 
\[
\frac{\partial G}{\partial p_{k}}=\partial_{p_{k}}(\tilde{g}_{ij}(\bar{p})\bar{v}^{i}\bar{v}^{j})=0,
\]
then the differential of the determinant of the metric tensor $\tilde{g}_{ij}(p)$
with respect to the base point $p$, evaluated at $\bar{p}$, also
vanishes, i.e., 
\[
D_{p}(\det([\tilde{g}]))(\bar{p})=0
\]
 with $\textrm{\textrm{sign}}(\tilde{g})=(1,m-1)$.

\textbf{~}\\ Now, we examine the $p$--derivatives. From the definition,
we have 

\begin{equation}
(D_{p}\tilde{g})(\bar{p})(\bar{v},\bar{v})=\frac{\partial\tilde{g}_{ij}}{\partial p^{k}}(\bar{p})\bar{v}^{i}\bar{v}^{j}=0.\label{eq:p-derivatives}
\end{equation}
Let the matrix representation of the metric $\tilde{g}$ in radical-adapted
Gauss-like coordinates (Definition~\ref{def:Radical-Adapted-Gauss-Like-Coordinates})
be given by the block matrix
\[
[\tilde{g}_{p}]=\left(\begin{array}{cc}
\lambda_{1} & 0\\
0 & \tilde{G}(p)
\end{array}\right),
\]
\textbf{~}\\ where $\tilde{G}$ is the $(m-1)\times(m-1)$ submatrix
representing the purely spatial part of the metric. If we assume that
at a point $\bar{p}$ we have $\lambda_{1}(\bar{p})=0$, then for
a vector $\bar{v}=(v_{1},\hat{v})$, we can compute the square of
its norm:

\[
\tilde{g}_{\bar{p}}(\bar{v},\bar{v})=\lambda_{1}(\bar{p})(v^{1})^{2}+\tilde{G}(\bar{p})(\hat{v},\hat{v}).
\]
With the assumption $\lambda_{1}(\bar{p})=0$, the condition $\tilde{g}_{\bar{p}}(\bar{v},\bar{v})=0$
then implies $\tilde{G}(\bar{p})(\hat{v},\hat{v})=0$. Since $\tilde{G}$
is Riemannian metric, it is positive definite. Therefore, this condition
holds if and only if $\hat{v}=0$ for the vector $\hat{v}$. From
Equation~\ref{eq:p-derivatives} we obtain 

~
\[
D_{p}\lambda_{1}(\bar{p})\underset{\neq0}{\underbrace{\bar{v}_{1}^{2}}}+D_{p}\hat{g}(\hat{v},\hat{v})=0\Longrightarrow D_{p}\lambda_{1}(\bar{p})=0.
\]
Consequently, we have

\[
D_{p}(\det([\tilde{g}]))(\bar{p})=D_{p}(\det(\lambda_{1})\det([\tilde{G}]))(\bar{p})=\underset{=0}{\underbrace{D_{p}\lambda_{1}(\bar{p})}}\det(\tilde{G})(\bar{p})+\underset{=0}{\underbrace{\lambda_{1}(\bar{p})}}D_{p}(\det(\tilde{G})(\bar{p}))=0
\]
for $\bar{p}\in\mathcal{H}$. 
\end{proof}
Finally, this allows us to prove Proposition~\ref{prop:pseudo-Euclidean embedding}:
\begin{proof}
Let $(\tilde{M},\tilde{g})$ be an $m$--dimensional transverse signature-type
changing manifold. By Lemma~\ref{lem:Lightcone regular}, it follows
immediately that $(\tilde{M},\tilde{g})$ is LC (light-cone)-regular.
According to Proposition 4.9. in~\cite{Bernig}, the LC-regularity
of $(\tilde{M},\tilde{g})$, ensures the existence of a non-singular
pseudo-Riemannian manifold $(M',g')$ of specific dimension $(D)$
and signature $(q,p)$ into which $(\tilde{M},\tilde{g})$ can be
isometrically embedded. Let $\phi_{1}:\tilde{M}\hookrightarrow M'$
denote this isometric embedding. Furthermore, by the pseudo-Riemannian
Nash-type embedding theorem~\cite{Clarke. On the global isometric embedding of pseudo-Riemannian manifolds},
any non-singular pseudo-Riemannian manifold $(M',g')$ of dimension
$D$ and signature $(q,p)$ admits a global isometric embedding into
a pseudo-Euclidean space $\mathbb{R}^{Q,P}$ (i.e., a flat space with
a constant metric of signature $(Q,P)$) for sufficiently large $P$
and $Q$. Let $\phi_{2}:M'\hookrightarrow\mathbb{R}^{Q,P}$ denote
this second isometric embedding. The exact values of $P$ and $Q$
depend on $D,p,q$ and the specific version of the embedding theorem
used. Finally, the composition of two isometric embeddings is itself
an isometric embedding. Thus, the composite map $\Phi=\phi_{2}\circ\phi_{1}:\tilde{M}\hookrightarrow\mathbb{R}^{Q,P}$
provides an isometric embedding of $(\tilde{M},\tilde{g})$ into the
$(P+Q)$--dimensional pseudo-Euclidean space $\mathbb{R}^{Q,P}$.
This completes the proof. 
\end{proof}
\begin{rem}
\label{rem:geometric interpretation braneworlds} It is important
to emphasize that this framework represents a specific geometric interpretation
rather than a general equivalence to established braneworld theories.
This distinction is particularly significant, as we are no longer
dealing with classical spacetimes, but rather with signature-type
changing manifolds serving as branes, and pseudo-Euclidean manifolds
that, while serving as the bulk, are not necessarily Lorentzian. These
bulk manifolds are treated as solutions to the higher-dimensional
vacuum Einstein field equations. If the bulk has a non-degenerate
$(q,p)$ metric with $q>1$, then this framework is mathematically
possible but raises profound physical questions:\footnote{Mathematically, this imposes constraints on how the brane is embedded
in the bulk: out of the $q$ negative eigenvalues, only one can align
with the brane's intrinsic time direction.} A non-Lorentzian bulk of this kind would be $q$--temporal rather
than strictly Lorentzian, introducing interesting yet potentially
problematic physical consequences. The presence of multiple timelike
directions could lead to causality violations, such as closed timelike
curves (CTCs) or indefinite causal ordering, as well as issues related
to dynamical stability. Fields propagating through such a bulk may
be influenced by the additional timelike dimensions, resulting in
challenges with predictability and the well-posedness of initial value
problems. To make this scenario physically meaningful, a consistent
mechanism would be required to mitigate these difficulties, such as
imposing constraints on field propagation within the bulk or defining
specific interactions between the brane and the extra timelike dimensions.
Remarkably, recent studies~\cite{3 + 2 cosmology: Unifying FRW metrics in the bulk}
have explored models in which a spatially flat FRW metric is embedded
in a bulk with a $(2,3)$ signature (i.e., possessing two timelike
dimensions), yielding a unique bulk solution referred to as the \textquotedblleft M-metric\textquotedblright .
\end{rem}

\begin{example}
Consider the classic example of a spacetime $\tilde{M}$ with signature-type
change, obtained by cutting an $S^{4}$ (a Riemannian manifold) along
its equator and smoothly joining it to the corresponding half of a
de Sitter space (a Lorentzian manifold). The signature changes from
Euclidean to Lorentzian across the equator where $t=0$. This is the
universe model that satisfies the Hartle and Hawking \textquotedblleft no
boundary\textquotedblright{} condition~\cite{Hartle Hawking - Wave function of the Universe}
and is a $(3+1)$--dimensional transverse signature-type changing
manifold with a transverse radical. Therefore, according to Proposition~\ref{prop:pseudo-Euclidean embedding},
the manifold $\tilde{M}$ can be isometrically embedded in a $(4+d)$--dimensional,
non-singular pseudo-Euclidean $(M,g)$, thereby realizing a braneworld
scenario for the \textquotedblleft no boundary\textquotedblright{}
model.
\end{example}

\section{$\mathcal{H}$--Global Isometric Embedding}

Building upon the general existence of an isomteirc embedding (Proposition~\ref{prop:pseudo-Euclidean embedding}),
our primary contribution is the demonstration of a significantly stronger
result: the existence of an $\mathcal{H}$--global isometric embedding.
We introduce this notion and prove its existence, offering a more
constrained and powerful form of embedding.

\subsection{$\mathcal{H}$--Global Isometric Embedding into Minkowski Spacetime}

While Section~\ref{sec:Signature-type-Changing-Braneworlds} demonstrated
the existence of a general isometric embedding into a pseudo-Euclidean
manifold, we now strengthen this result by proving with Theorem~\ref{thm:H-Global Minkowski}
the existence of an $\mathcal{H}$--global isometric embedding into
Minkowski spacetime. This constitutes an important special case within
the class of pseudo-Euclidean manifolds.
\begin{thm*}
Let $(\tilde{M},\tilde{g})$ be an $n$--dimensional transverse signature-type
changing manifold with a transverse radical, and fix an $\mathcal{H}$--global
neighborhood $U=\bigcup_{q\in\mathcal{H}}U(q)\subset\tilde{M}$. Then
there exists a sufficiently large integer $N$ and an $\mathcal{H}$--global
isometric embedding $\psi:U\to\mathbb{R}^{1,N-1}$ of $U$ (with the
induced metric $\tilde{g}|_{U}$) into Minkowski spacetime $(\mathbb{R}^{1,N-1},\eta)$.
\end{thm*}
\begin{proof}
Since $(\tilde{M},\tilde{g})$ is an $n$--dimensional transverse
signature-type changing manifold with a transverse radical, by Definiton~\ref{def:Radical-Adapted-Gauss-Like-Coordinates}
and Definiton~\ref{def:H-Global}, there exists an $\mathcal{H}$--global
neighborhood $U=\bigcup_{q\in\mathcal{H}}U(q)\text{\ensuremath{\subset}}\tilde{M}$.
On each $U(q)$ we can define radical-adapted Gauss-like coordinates
$(t,\mathbf{\hat{x}}):=(t,x^{1},\ldots,x^{n-1})\in\mathbb{R}\times\mathbb{R}^{n-1}$
such that the metric tensor $\tilde{g}$ takes the canonical form
\[
\tilde{g}=-t(dt)^{2}+g_{ij}(t,x^{1},\ldots,x^{n-1})dx^{i}dx^{j},
\]
where the indices $i,j$ range from $1$ to $n-1$. \textbf{~}\\ Within
$U$, we can express the metric components of the matrix $[\tilde{g}_{\mu\nu}]$
in terms of two block matrices, $[f]$ and $[h]$, by algebraically
decomposing the $g_{tt}$ component:

~

 $[\tilde{g}_{\mu\nu}]=\left (\begin{array}{@{}c|ccc@{}}      -(t+1) & 0 & \cdots & 0   \\\hline     0 &  &  &   \\     \vdots &  & \mathbf{0} &  \\     0 &  &  &      \end{array}\right)
+ \left (\begin{array}{@{}c|ccc@{}}      1 & 0 & \cdots & 0   \\\hline     0 &  &  &   \\     \vdots &  & \tilde{G} &  \\     0 &  &  &      \end{array}\right)=[f]+[h].$

\textbf{~}

\textbf{~}\linebreak{}
Here, $\mathbf{0}$ denotes an $(n-1)\times(n-1)$ zero matrix, and
$\tilde{G}=[g_{ij}(t,x^{1},\ldots,x^{n-1})]$ is the $(n-1)\times(n-1)$
submatrix representing the purely spatial part of the metric. This
decomposition separates the $t$--dependent term and a constant offset
in the $g_{tt}$.

~

Due to the Nash embedding theorem for Riemannian manifolds~\cite{Nash},
for the $n$--dimensional Riemannian manifold $(\tilde{M},h)$, there
exists a number $N_{h}$ and an isometric embedding $\Phi:U\subseteq\tilde{M}\longrightarrow\mathbb{R}^{N_{h}}$,
such that for every point $p\in U$, the derivative $D\Phi_{p}$ is
a linear map from the tangent space $T_{p}\tilde{M}$ to $\mathbb{R}^{N_{h}}$,
which is compatible with the given inner product on $T_{p}\tilde{M}$
and the standard dot product of $\mathbb{R}^{N_{h}}$. That is, for
all tangent vectors $u,v\in T_{p}\tilde{M}$:

\[
h(u,v)=D\Phi_{p}(u)\cdot D\Phi_{p}(v).
\]
We choose the total dimension $N$ of the target Minkowski space $(\mathbb{R}^{1,N-1},\eta)$
such that its spatial dimension satisfies $N-1\geq N_{h}$. 

~

Let $X=\tau\partial_{t}+w^{i}\partial_{x^{i}}\in T_{(t,\hat{x})}\mathbb{R}\oplus\mathbb{R}^{N_{h}}$
be a tangent vector at a point $(t,\mathbf{\hat{x}})$ on $U\subseteq\tilde{M}$,
where $\tau\in\mathbb{R}$ is the component along $\partial_{t}$
and $\mathbf{\hat{w}}=(w^{1},\ldots,w^{n-1})$ represents the components
along the spatial coordinate directions $\partial_{x^{i}}$. The isometric
embedding property of $\Phi$ (applied to the Riemannian metric $h$)
implies that the squared norm of the pushforward of $X$ under $\Phi$
is given by the metric $h(X,X)$: $$ \|D\Phi_{(t,\mathbf{\hat{x}})}(X)\|_{\mathbb{R}^{N_{h}}}^2 = h(X,X).$$

\textbf{~}\\ Expanding $h(X,X)$ using the definition of the metric
\[
h=(dt)^{2}+g{ij}(t,x^{1},\dots,x^{n-1})dx^{i}dx^{j}
\]
 yields:

\[
h(X,X)=h(\tau\partial_{t}+w^{i}\partial_{x^{i}},\tau\partial_{t}+w^{j}\partial_{x^{j}})=\tau^{2}h(\partial_{t},\partial_{t})+2\tau w^{i}h(\partial_{t},\partial_{x^{i}})+w^{i}w^{j}h(\partial_{x^{i}},\partial_{x^{j}}).
\]
Since $h(\partial_{t},\partial_{t})=1$, $h(\partial_{t},\partial_{x^{i}})=0$,
and $h(\partial_{x^{i}},\partial_{x^{j}})=g_{ij}$, this simplifies
to:

\[
h(X,X)=\tau^{2}+\underset{g(t,\mathbf{\hat{x}})(w,w)}{\underbrace{g_{ij}(t,\mathbf{\hat{x}})w^{i}w^{j}}}.
\]
Here, $g_{ij}(t,\hat{x})w^{i}w^{j}$ denotes the quadratic form of
the spatial metric $\tilde{G}$ applied to the vector $\mathbf{\hat{w}}$.
Next, we define the embedding function $\psi:U\subseteq\tilde{M}\longrightarrow\mathbb{R}^{1,N_{h}}$
by:

\[
\psi(t,\mathbf{\hat{x}}):=(f(t),\Phi(t,\mathbf{\hat{x}}))=(f(t),\Phi^{1}(t,\mathbf{\hat{x}}),\ldots,\Phi^{N_{h}}(t,\mathbf{\hat{x}}))\in\mathbb{R}\oplus\mathbb{R}^{N_{h}},
\]
where $f:\mathbb{R\longrightarrow\mathbb{R}}$ is a smooth function
to be determined, and $\Phi:U\subseteq\tilde{M}\longrightarrow\mathbb{R}^{N_{h}}$
is the Nash embedding for the Riemannian metric $h$. Let $X=(\tau,w^{1},\text{\dots},w^{n\text{\textminus}1})=(\tau,\mathbf{\hat{w}})\in T_{(t,\mathbf{\hat{x}})}\mathbb{R}\oplus\mathbb{R}^{N_{h}}$
be a tangent vector on $U\subseteq\tilde{M}$ at $(t,\mathbf{\hat{x}})$.
We calculate the induced metric $\psi^{*}\eta$ by evaluating $\eta(D\psi(X),D\psi(X))$.
The pushforward of $X$ by $\psi$ is:

\[
D\psi(X)=(X(f(t)),X(\Phi^{1}(t,\mathbf{\hat{x}})),\ldots,X(\Phi^{N_{h}}(t,\mathbf{\hat{x}}))=(f'(t)\tau,D\Phi(X)).
\]
Now, we evaluate this vector with respect to the Minkowski metric
$\eta=-(dy^{0})^{2}+\sum_{\alpha=1}^{N-1}(dy^{\alpha})^{2}$ in the
target space (where $y^{0}$ corresponds to $f(t)$ and $y^{\alpha}$
for $\alpha>0$ corresponds to the components of $\Phi(t,\mathbf{\hat{x})}$):

~

$\psi^{\eta}(X,X)$

$=\eta(D\psi(X),D\psi(X))=\eta((f'(t)\tau,D\Phi(X)),(f'(t)\tau,D\Phi(X)))$

$=-(f'(t)\tau)^{2}+\left\Vert D\Phi(X)\right\Vert _{R^{N_{h}}}^{2}$

$=-(f'(t))^{2}\tau^{2}+h(X,X)\quad\text{(by the isometric property of }\Phi\text{ for metric }h)$

$=-(f'(t))^{2}\tau^{2}+(\tau^{2}+g_{ij}(t,\mathbf{\mathbf{\hat{x}}})w^{i}w^{j})\quad\text{(from previous calculation of }h(X,X))$

$=(1-(f'(t))^{2})\tau^{2}+g_{ij}(t,\mathbf{\hat{x}})w^{i}w^{j}$.

\textbf{~}\\ For $\psi$ to be an isometric embedding, the induced
metric $\psi^{*}\eta$ must match the original metric $\tilde{g}$
on $U\subseteq\tilde{M}$. That is, $\psi^{*}\eta(X,X)=\tilde{g}(X,X)$.
We know $\tilde{g}(X,X)=-t\tau^{2}+g_{ij}(t,\hat{x})w^{i}w^{j}$.
By comparing the coefficients of $\tau^{2}$ and $w^{i}w^{j}$, we
require $1-(f'(t))^{2}=-t$. This directly implies:

\[
\ensuremath{(f'(t))^{2}=1+t}.
\]
Taking the square root, we get $f'(t)=\pm\sqrt{1+t}$. Without loss
of generality, we can choose the negative sign for $f'(t)$, which
gives $f'(t)=-\sqrt{1+t}$. Integrating this, we find the function 

\[
f(t)=-\int\sqrt{1+t}dt=-\frac{2}{3}(1+t)^{\frac{3}{2}}+C,
\]
where $C$ is an integration constant. This constant can be set to
zero by choosing an appropriate origin for the time coordinate in
the embedding space. This condition determines the \textit{temporal function} $f(t)=-\frac{2}{3}(1+t)^{\frac{3}{2}}$,
which is \textit{strictly monotonic decreasing}. This choice of $f(t)$
yields the $\mathcal{H}$--global isometric embedding $\psi:U\subseteq\tilde{M}\longrightarrow\mathbb{R}^{1,N_{h}}$:

\begin{equation}
\psi(t,\mathbf{\hat{x}}):=\left(-\frac{2}{3}(1+t)^{\frac{3}{2}},\Phi(t,\mathbf{\hat{x}})\right)=\left(-\frac{2}{3}(1+t)^{\frac{3}{2}},\Phi^{1}(t,\mathbf{\hat{x}}),\ldots,\Phi^{N_{h}}(t,\mathbf{\hat{x}})\right)\in\mathbb{R}\oplus\mathbb{R}^{N_{h}}.\label{eq:H-Global Minskwoski embedding}
\end{equation}
This concludes the proof of the existence of the $\mathcal{H}$--global
isometric embedding of $(\tilde{M},\tilde{g})$ into Minkowski space.
\end{proof}

\subsection{$\mathcal{H}$--Global Isometric Embedding into Misner Space}

Next, we introduce an extended geometrical perspective on braneworld
theory by considering a signature-type changing manifold as the embedded
space and Misner space~\cite{Misner} as the higher-dimensional flat
space. The latter naturally possesses the topological structure $\mathbb{R}^{N}\times S^{1}$,
where one dimension is compact due to an identification.

\subsubsection{Misner Orbifold}

The Misner orbifold, also known as the Lorentzian orbifold $\mathbb{R}^{1,1}/\text{boost}$,
is a space that can be understood as Minkowski space modulo discrete
group actions. Formally, in arbitrary dimension, this is $\mathbb{R}^{1,N-1}/\Gamma$,
where $\Gamma\subset O(1,N-1)$ is a discrete subgroup of Lorentzian
isometries. This group $\Gamma$ is generated by the one-parameter
pure Lorentz transformation (hyperbolic rotation) $A:\mathbb{R}^{1,N-1}\longrightarrow\mathbb{R}^{1,N-1}$
given by:

\[
A(\tau,y^{1},y^{2},\ldots,y^{N-1})=(\tau\cosh(\pi)+y^{1}\sinh(\pi),\tau\sinh(\pi)+y^{1}\cosh(\pi),y^{2},\ldots,y^{N-1}).
\]
Consequently, the group $\Gamma$ consists of all integer powers of
$A$, i.e., $\Gamma=\{A^{n}\mid n\in\mathbb{Z}\}$. This means points
are identified according to the relation:

\[
(\tau,y^{1},y^{2},\ldots,y^{N-1})\sim(\tau\cosh(n\pi)+y^{1}\sinh(n\pi),\tau\sinh(n\pi)+y^{1}\cosh(n\pi),y^{2},\ldots,y^{N-1})
\]
for all $n\in\mathbb{Z}$. 

~

Specifically, the $N$--dimensional Misner orbifold is formed conceptually
from Minkowski spacetime $(\mathbb{R}^{1,N-1},\eta)$ by quotienting
by a discrete one-parameter subgroup of isometries. This group, denoted
by $\Gamma$, exclusively affects the $(\tau,y^{1})$--subspace and
is generated by a Lorentz boost matrix $A(\theta_{0})$ for some fixed
rapidity $\theta_{0}$: 

~
\[
A(\theta_{0})=\begin{pmatrix}\cosh\theta_{0} & \sinh\theta_{0} & 0 & \cdots & 0\\
\sinh\theta_{0} & \cosh\theta_{0} & 0 & \cdots & 0\\
0 & 0 & 1 & \cdots & 0\\
\vdots & \vdots & \vdots & \ddots & \vdots\\
0 & 0 & 0 & \cdots & 1
\end{pmatrix}
\]
\textbf{~}\\ The remaining spectator directions, $y^{2},\ldots,y^{N-1}$,
remain unaffected, which justifies the generalization to $N$ dimensions.
When Minkowski spacetime is subjected to this group action, it is
naturally foliated by the orbits of $\Gamma$. These orbits lie in
the $(\tau,y^{1})$--plane and are $1$--dimensional hyperbolas
of constant $s^{2}=-\tau^{2}+(y^{1})^{2}$. Thus, $(\mathbb{R}^{1,N-1},\eta)$
is foliated by these $1$--dimensional boost orbits in each $(\tau,y^{1})$--plane
slice at fixed values of $y^{2},\ldots,y^{N-1}$. If the entire Minkowski
spacetime were to be formally quotiented by $\Gamma$, this foliation
would descend to a singular foliation on the resulting orbifold, remaining
regular everywhere except at the fixed point (the origin). This induces
a stratified structure, with strata corresponding to different orbit
types. However, such a global quotient of the entire Minkowski spacetime
by $\Gamma$ yields a non-Hausdorff space. Having established the
properties of this orbital structure, we can now prove Proposition~\ref{prop:embedding transverse to the orbital foliation}.
\begin{prop*}
Let $(\tilde{M},\tilde{g})$ be an $n$--dimensional transverse signature-type
changing manifold with a transverse radical, with coordinates $(t,\mathbf{\hat{x}})=(x^{1},\text{\dots},x^{n\text{\textminus}1})$.
Let $\psi:U\subseteq\tilde{M}\longrightarrow\mathbb{R}^{1,N-1}$,
\[
\psi(t,\mathbf{\hat{x}}):=\left(-\frac{2}{3}(1+t)^{\frac{3}{2}},\Phi^{1}(t,\mathbf{\hat{x}}),\ldots,\Phi^{N-1}(t,\mathbf{\hat{x}})\right)
\]
be the embedding given in Theorem~\ref{thm:H-Global Minkowski}.
Assume the image of $\psi$ lies in the region 
\[
\mathcal{R}:=\{(\tau,y^{1},\dots,y^{N-1})\in\mathbb{R}^{1,N-1}\;|\;y^{1}-\tau>0\}.
\]
Then the image $\mathrm{Im}(\psi)$ is transverse to the orbital foliation
generated by the one-parameter subgroup of hyperbolic rotations $\Gamma\subset O(1,N-1)$
(acting in the $(\tau,y^{1})$--plane), provided that the spatial
tangent vectors of the embedding are non-zero.\footnote{Non-zero in the following sense: the spatial projections $\partial_{t}\Phi,\;\partial_{x^{1}}\Phi,\dots,\partial_{x^{\,n-1}}\Phi\in\mathbb{R}^{N-1}$
do not all lie in the one-dimensional subspace spanned by $e_{1}=(1,0,\dots,0)\in\mathbb{R}^{N-1}$.}
\end{prop*}
\begin{proof}
To demonstrate that the image of the embedding $\psi(t,\mathbf{\hat{x}})$
is transversal to the foliation by orbits in Minkowski space, we must
prove that, at every point of intersection, the tangent space of the
embedded manifold is not contained within the tangent space of the
orbit's leaf. 

The tangent space of the orbits is spanned by the Killing vector
field for the hyperbolic rotation. Recall that the Killing field generating
hyperbolic rotations in the $(\tau,y^{1})$--plane is 
\[
K\;=\;\tau\,\frac{\partial}{\partial y^{1}}+y^{1}\,\frac{\partial}{\partial\tau},
\]
 whose spatial part (projection to the $y$--coordinates) is $(\tau,0,\dots,0)\in\mathbb{R}^{N-1}$.

The tangent space $T_{\psi(p)}\mathrm{Im}(\psi)$ at a point $p=(t,\mathbf{\hat{x}})\in U\subseteq\tilde{M}$
is spanned by the partial derivatives
\[
\frac{\partial\psi}{\partial t}=\Big(-\sqrt{1+t},\;\partial_{t}\Phi^{1},\dots,\partial_{t}\Phi^{N-1}\Big),\qquad\frac{\partial\psi}{\partial x^{i}}=\Big(0,\;\partial_{x^{i}}\Phi^{1},\dots,\partial_{x^{i}}\Phi^{N-1}\Big),
\]
 for $i=1,\dots,n-1$. Here we use $\partial_{t}(-\tfrac{2}{3}(1+t)^{3/2})=-\sqrt{1+t}$
and $t>-1$ on $\mathcal{R}$.

To prove transversality it suffices to show that the Killing field
$K$ is never tangent to $\mathrm{Im}(\psi)$; equivalently, there
are no scalars $a,b_{1},\dots,b_{n-1}$ (depending on the point) such
that 
\[
K\;=\;a\,\frac{\partial\psi}{\partial t}\;+\;\sum_{i=1}^{n-1}b_{i}\,\frac{\partial\psi}{\partial x^{i}}.
\]
Project this vector equality to the spatial coordinates by the projection
$\pi:\mathbb{R}^{1,N-1}\to\mathbb{R}^{N-1}$. Since $\pi(K)=(\tau,0,\dots,0)$
and $\pi(\partial\psi/\partial t)=\partial_{t}\Phi$, $\pi(\partial\psi/\partial x^{i})=\partial_{x^{i}}\Phi$,
we obtain the linear system in $\mathbb{R}^{N-1}$ 
\[
(\tau,0,\dots,0)\;=\;a\,\partial_{t}\Phi\;+\;\sum_{i=1}^{n-1}b_{i}\,\partial_{x^{i}}\Phi.\tag{1}
\]
Write (1) componentwise. For the first spatial component ($y^{1}$)
we have 
\[
\tau\;=\;a\,\partial_{t}\Phi^{1}+\sum_{i=1}^{n-1}b_{i}\,\partial_{x^{i}}\Phi^{1},\tag{2}
\]
and for each remaining spatial component $A=2,\dots,N-1$, 
\[
0\;=\;a\,\partial_{t}\Phi^{A}+\sum_{i=1}^{n-1}b_{i}\,\partial_{x^{i}}\Phi^{A}.\tag{3}
\]

Consider the homogeneous linear system given by the $A\ge2$ equations
(3). This is a system of $N-2$ linear equations in the $n$ unknowns
$(a,b_{1},\dots,b_{n-1})$. By the hypothesis on the spatial tangents,
the set 
\[
\{\partial_{t}\Phi,\;\partial_{x^{1}}\Phi,\dots,\partial_{x^{\,n-1}}\Phi\}
\]
does not lie entirely in the $e_{1}$--direction; therefore their
components in indices $A=2,\dots,N-1$ are not all simultaneously
zero. Consequently the homogeneous system (3) forces 
\[
a=0,\qquad b_{1}=\cdots=b_{n-1}=0
\]
only in the case where the $A\ge2$ components of every generator
vanish. If, as in our hypothesis, not all spatial-projection vectors
are collinear with $e_{1}$, then (3) implies that the unique solution
is $a=b_{1}=\cdots=b_{n-1}=0$. But $a=b_{i}=0$ contradicts (2),
because $\tau=-\tfrac{2}{3}(1+t)^{3/2}\neq0$ for $t>-1$. Thus no
choice of scalars $a,b_{i}$ can make $K$ a linear combination of
the tangent vectors to $\mathrm{Im}(\psi)$ at any point in the region
$\mathcal{R}$.

~

Therefore $K$ is nowhere tangent to $\mathrm{Im}(\psi)$; equivalently
the tangent space $T_{\psi(p)}\mathrm{Im}(\psi)$ is never contained
in the tangent space to the orbit through $\psi(p)$. Hence $\mathrm{Im}(\psi)$
is transverse to the orbital foliation generated by $\Gamma$, as
claimed. 
\end{proof}

\begin{rem*}
The hypothesis that the spatial projections $\partial_{t}\Phi,\partial_{x^{i}}\Phi$
do not all lie in the span of $e_{1}$ can be replaced by the slightly
stronger but more explicit requirement that the $(N-2)\times n$ matrix
\[
\big(\partial_{t}\Phi^{A},\;\partial_{x^{1}}\Phi^{A},\dots,\partial_{x^{\,n-1}}\Phi^{A}\big)_{A=2,\dots,N-1}
\]
 have rank at least $1$ (i.e. not be the zero matrix). Under that
linear algebra condition the argument above is immediate. 
\end{rem*}

\begin{lem}
\label{lem:orbit-tangency-H-global} Let $\psi(U)\subset\mathcal{R}\subset\mathbb{R}^{1,N-1}$
be the embedded image of a fixed $\mathcal{H}$--global neighborhood
$U\subseteq\tilde{M}$, and let $\Gamma$ be the one-parameter boost
group acting on $\mathcal{R}$ with generator $K$. Fix an orbit 
\[
\mathcal{O}=\{\gamma_{s}(q):s\in\mathbb{R}\},
\]
which meets $\psi(U)$ in at least two distinct points $\gamma_{s_{1}}(q)$
and $\gamma_{s_{2}}(q)$ with $s_{1}<s_{2}$.

Assume the composed map 
\[
F_{q}:(s_{1},s_{2})\longrightarrow\mathbb{R},\quad F_{q}(s):=t\big(\psi^{-1}(\gamma_{s}(q))\big)
\]
is smooth on $(s_{1},s_{2})$ (this holds whenever $\psi^{-1}$ is
smooth on its image, which is true since $\psi$ is an embedding defined
on $U$). Then exactly one of the following holds:

\begin{enumerate}
\item  \(F_q\) is strictly monotone on \((s_1, s_2)\), or
\item  \(F_q\) attains an interior extremum at some \(s_* \in (s_1, s_2)\), i.e., \(F_q'(s_*) = 0\), which is equivalent to the statement that the generator \(K\) is tangent to the \(t\)--level set (and hence tangent to \(\psi(U)\)) at the point \(\gamma_{s_*}(q)\). 
\end{enumerate} 
\end{lem}

\begin{proof}
The function $F_{q}$ is smooth on the open interval $(s_{1},s_{2})$.
Suppose $F_{q}$ is not strictly monotone. Then, by standard results
in calculus, it must attain either a local maximum or a local minimum
at some interior point $s_{*}\in(s_{1},s_{2})$. At such an extremum,
we have 
\[
F_{q}'(s_{*})=0.
\]
By the definition of $F_{q}$ and the chain rule, 
\[
F_{q}'(s_{*})=dt_{p_{*}}\big((\psi^{-1})_{*}(K|_{\gamma_{s_{*}}(q)})\big),
\]
 where $p_{*}=\psi^{-1}(\gamma_{s_{*}}(q))$.

Thus, $F_{q}'(s_{*})=0$ exactly when the pushforward of the Killing
vector field $K$ under $\psi^{-1}$ at $\gamma_{s_{*}}(q)$ lies
in the kernel of the differential $dt_{p_{*}}$, i.e., 
\[
(\psi^{-1})_{*}(K)\in\ker(dt_{p_{*}}).
\]
Geometrically, this means that the infinitesimal generator $K$ of
the orbit is tangent to the level set of the absolute time function
$\mathfrak{h}(t,\mathbf{\hat{x}}):=t$ through the point $p_{*}$,
or equivalently, tangent to $\psi(U)$ at $\gamma_{s_{*}}(q)$. This
yields the claimed dichotomy.
\end{proof}

\begin{cor}
\label{cor:injectivity-by-lemma-orbit-tangency}Under the hypotheses
of Theorem~\ref{thm:H-Global Minkowski}, if for every orbit $\mathcal{O}$
that meets $\psi(U)$ in at least two points the corresponding function
$F_{q}$ is \emph{not} strictly monotone on the intervening interval,
then transversality (Proposition~\ref{prop:embedding transverse to the orbital foliation})
rules out the second alternative of Lemma~\ref{lem:orbit-tangency-H-global}.
Hence no orbit can meet $\psi(U)$ in two distinct points with different
$t$--values.
\end{cor}

\subsubsection{Misner Space}

To construct the actual, well-behaved Misner space (which is a Lorentzian
manifold), it must therefore be defined as the quotient of a specific,
restricted region of Minkowski spacetime:
\begin{defn}
\label{def:Misner region}Let $(\mathbb{R}^{1,N-1},\eta)$ denote
the $N$--dimensional Minkowski spacetime, endowed with coordinates
$(\tau,y^{1},y^{2},\ldots,y^{N-1})$ and the flat Lorentzian metric
$\eta_{\mu\nu}=\textrm{diag}(-1,1,\ldots,1)$. The region $\mathcal{R}\subset\mathbb{R}^{1,N-1}$
is defined as the open half-space: 
\[
\mathcal{R}:=\{(\tau,y^{1},y^{2},\ldots,y^{N-1})\in\mathbb{R}^{1,N-1}\mid y^{1}-\tau>0\}.
\]
\ 
\end{defn}

By restricting the base space to a suitable region $\mathcal{R}\subset\mathbb{R}^{1,N-1}$,
the group action $\Gamma$ becomes properly discontinuous and free
on that region, leading to a regular quotient manifold, known as Misner
space.
\begin{defn}
\label{Misner space}An $N$--dimensional Misner space is defined
as the quotient space $(\mathcal{R}/\Gamma,g)$, where $g$ is the
induced metric on the quotient space and $\text{\textgreek{G}\ensuremath{\subset}}O(1,N\text{\textminus}1)$
a discrete group of Lorentz isometries which acts on $\mathcal{R}$.
The action of $\Gamma$ on $\mathcal{R}$ is properly discontinuous
and free.
\end{defn}

Alternatively, Misner space can be described using coordinates $(T,\phi,y^{2},\ldots,y^{N-1})$,
by starting with Minkowski space $(\mathbb{R}^{1,N-1},\eta)$, where
the metric structure is expressed as $\eta=-(d\tau)^{2}+\sum_{i=1}^{N-1}(dy^{i})^{2}$,
and then applying a sequence of coordinate transformations (see~\cite{Ori Misner space,Rieger + Throne}
for more details). In these Misner coordinates, the metric $g$ takes
the form
\begin{equation}
g=-2dTd\phi-T(d\phi)^{2}+\sum_{i=2}^{N-1}(dy^{i})^{2},\label{eq:Misner metric}
\end{equation}
with coordinate domains $-\infty<T<\infty$ and $0\leq\phi\leq2\pi$.
Here, $T$ remains a timelike coordinate, while $\phi$ serves as
an angular coordinate and $y^{2},\ldots,y^{N-1}$ are spatial spectator
directions. Thus, the underlying topology of $N$--dimensional Misner
space is that of a hypercylinder $\mathbb{R}^{N-1}\times S^{1}$.
In Misner coordinates, the orbits (generated by the Lorentz group
$\Gamma$) in Misner space $\mathcal{M}_{\text{Misner}}$ are curves
confined to surfaces of constant $T$, where only the angular-like
coordinate $\phi$ changes. These orbits are closed timelike curves
(CTCs) when $T>0$, or hyperbolic orbits when $T<0$. Thus, $\mathcal{M}_{\text{Misner}}$
is the set of all these orbits. Each unique point in $\mathcal{M}_{\text{Misner}}$
is an \textit{entire} orbit (an equivalence class) from $\mathcal{R}\subset\mathbb{R}^{1,N-1}$.

~

The region $\mathcal{R}$ of $N$--dimensional Minkowski space $(\mathbb{R}^{1,N-1},\eta)$
serves as the covering space for $N$--dimensional Misner space $(\mathcal{R}/\Gamma,g)$,
with the covering map being the natural projection $\pi:\mathcal{R}\to\mathcal{R}/\Gamma$.
This generalizes the $2$--dimensional case where region $I+II$
of Minkowski space is the covering space for $2$--dimensional Misner
space, as described in~\cite{Hawking + Ellis - The large scale structure of spacetime,Rieger - Misner}.
Moreover, the Misner coordinates $(T,\phi,y^{2},\ldots,y^{N-1})$
are related to the Minkowski coordinates $(\tau,y^{1},y^{2},\ldots,y^{N-1})$
by the transformations: 
\begin{equation}
T=\frac{1}{4}((y^{1})^{2}-\tau^{2}),\quad\phi=2\ln\left(\frac{y^{1}-t}{2}\right)\quad\text{and}\quad y^{i}=y^{i}\quad\text{for }i\in\{2,\ldots,N-1\}.\label{eq:Misner-Minkowski coordinates}
\end{equation}

\begin{rem}
\label{rem:spectator-coords} Note that for calculations, we can ignore
the spectator $y^{i}$--coordinates (for $i\in\{2,\ldots,N-1\}$)
in the Misner hypercylinder, as this is merely a trivial extension
of dimensionality. One can visualize the additional $y^{i}$--axes
as protruding from the $(\tau,y^{1})$--plane, providing a tangible
representation of the geometric relationships. Altogether, $\mathbb{R}^{N-1}\times S^{1}$
describes an $N$--dimensional space in which each point in $\mathbb{R}^{N-1}$
corresponds to an entire circle---imagine a (tiny) loop attached
to every point in the plane. Consequently, the space resembles a solid
structure with a compact, potentially ``hidden'' circular dimension
at each location, akin to an infinite stack of circles distributed
over an $(N-1)$--dimensional Euclidean space (``plane''). Without
loss of generality, we may fix the spectator coordinates $y^{i}$,
thereby reducing the hypercylinder to the well-known two-dimensional
Misner space. 
\end{rem}

Now consider the composition $\pi\circ\psi$, where $\pi:\mathcal{R}\to\text{\ensuremath{\mathcal{M}_{\text{Misner}}}}$
is the quotient map from a suitable region $\mathcal{R}\subset\mathbb{R}^{1,N-1}$,
and $\psi:U\subseteq\tilde{M}\longrightarrow\mathbb{R}^{1,N-1}$ is
an $\mathcal{H}$--global isometric embedding as defined in Theorem~\ref{thm:H-Global Minkowski}.
We are interested in the metric induced on $U\subseteq\tilde{M}$
by this composition. This requires the property of pullbacks for covariant
tensors.
\begin{lem}
\label{lem:Pullback-Property}Let $F:M\to N$ and $G:N\to P$ be local
diffeomorphisms. For any covariant tensor field $T$ on $P$, the
pullback $(G\circ F)^{*}T$ satisfies 
\[
(G\circ F)^{*}T=F^{*}(G^{*}T).
\]
\end{lem}

\begin{proof}
Let $T$ be a covariant tensor field of rank $k$ on $P$. Let $p\in M$
and $v_{1},\dots,v_{k}\in T_{p}M$. By the definition of the pullback
operation: 

\textbf{~}\\ \begin{align*} 
((G \circ F)^* T)_p(v_1, \dots, v_k) 
&= T_{(G \circ F)(p)}(d(G \circ F)_p(v_1), \dots, d(G \circ F)_p(v_k)) \\ 
&= T_{G(F(p))}((dG)_{F(p)}(dF)_p(v_1), \dots, (dG)_{F(p)}(dF)_p(v_k)) \end{align*} 

Now, consider $F^{*}(G^{*}T)$. Let $w_{i}=(dF)_{p}(v_{i})$, so $w_{i}\in T_{F(p)}N$.

\textbf{~}\\ \begin{align*} (F^* (G^* T))_p(v_1, \dots, v_k) &= (G^* T)_{F(p)}((dF)_p(v_1), \dots, (dF)_p(v_k)) \\ &= (G^* T)_{F(p)}(w_1, \dots, w_k) \\ &= T_{G(F(p))}((dG)_{F(p)}(w_1), \dots, (dG)_{F(p)}(w_k)) \\ &= T_{G(F(p))}((dG)_{F(p)}(dF)_p(v_1), \dots, (dG)_{F(p)}(dF)_p(v_k)) \end{align*}
~

Comparing the final expressions, we see that 
\[
((G\circ F)^{*}T)_{p}(v_{1},\dots,v_{k})=(F^{*}(G^{*}T))_{p}(v_{1},\dots,v_{k}),
\]
thus proving the identity. 
\end{proof}
This leads into the next proposition, where $\psi:U\subseteq\tilde{M}\longrightarrow\mathbb{R}^{1,N-1}$
is an $\mathcal{H}$--global isometric embedding, given by $\psi(t,\mathbf{\hat{x}})=\left(-\frac{2}{3}(1+t)^{3/2},\Phi(t,\mathbf{\hat{x}})\right)$.
We consider the case where the image $\psi(U)$ is contained within
the region $\mathcal{R}:=\{(\tau,y^{1},y^{2},\ldots,y^{N-1})\in\mathbb{R}^{1,N-1}\mid y^{1}-\tau>0\}$.
This condition, which is required for the subsequent construction
of the map to Misner space, is ensured by our specific choice of the
time coordinate and the corresponding properties of the spatial embedding
$\Phi(t,\mathbf{\hat{x}})$ dictated by the $\mathcal{H}$--global
isometric embedding theorem. We now state the proposition.
\begin{prop}
\label{prop:misner_immersion}Let $\psi:U\subseteq\tilde{M}\longrightarrow\mathbb{R}^{1,N-1}$
be an $\mathcal{H}$--global isometric embedding whose image is contained
within the region $\mathcal{R}:=\{(\tau,y^{1},y^{2},\ldots,y^{N-1})\in\mathbb{R}^{1,N-1}\mid y^{1}-\tau>0\}$.
Let $\pi:\mathcal{R}\to\mathcal{M}_{\text{Misner}}$ be the quotient
map to Misner space $\mathcal{M}_{\text{Misner}}$. Then the composition
$\pi\circ\psi$ is an isometric immersion.
\end{prop}

\begin{proof}
First, we establish that $\Psi=\pi\circ\psi$ is an immersion. The
map $\psi:U\subseteq\tilde{M}\to\mathbb{R}^{1,N-1}$ is an $\mathcal{H}$--global
isometric embedding, and thus a smooth immersion. The map $\pi:\mathcal{R}\to\mathcal{M}_{\text{Misner}}$
is a quotient map and a local diffeomorphism. Since the composition
of two immersions is an immersion, and a local diffeomorphism is an
immersion, it follows that $\Psi=\pi\circ\psi$ is a smooth immersion.

~

Next, we must prove that $\Psi$ is an isometry. This means we must
show that the pullback of the Misner space metric, $g_{\text{Misner}}$,
by the map $\Psi$ is equal to the original metric on $(\tilde{M},\tilde{g})$.
That is, we must prove: 
\[
\Psi^{*}g_{\text{Misner}}=\tilde{g}.
\]
We begin by recalling the relevant definitions and established results: 

~

(1) The map $\psi:U\subseteq\tilde{M}\to\mathbb{R}^{1,N-1}$ is an
isometric embedding. By definition, this means that the metric on
$(U,\tilde{g}|_{U})$, is the pullback of the flat Minkowski metric
$\eta$ on $\mathbb{R}^{1,N-1}$. Thus, $\tilde{g}=\psi^{*}\eta$. 

(2) The Misner space metric, $g_{\text{Misner}}$, is defined as the
\textit{unique} metric on the quotient space $\mathcal{M}_{\text{Misner}}$
such that the quotient map $\pi:\mathcal{R}\to\mathcal{M}_{\text{Misner}}$
is a local isometry. This implies that the pullback of the Misner
metric by $\pi$ is equal to the flat Minkowski metric on its domain
$\mathcal{R}$. Thus, $\pi^{*}g_{\text{Misner}}=\eta$. 

~

Now, let's evaluate the pullback of the Misner space metric by our
composite map $\Psi=\pi\circ\psi$. Using the functorial property
of pullbacks (Lemma~\ref{lem:Pullback-Property}), which
states that $(F\circ G)^{*}T=G^{*}(F^{*}T)$, we have: 
\[
\Psi^{*}g_{\text{Misner}}=(\pi\circ\psi)^{*}g_{\text{Misner }}=\psi^{*}(\pi^{*}g_{\text{Misner}}).
\]
Substituting the definition of the Misner metric from (2) into this
expression yields: 
\[
\Psi^{*}g_{\text{Misner}}=\psi^{*}\eta
\]
Finally, we substitute the definition of the isometric embedding $\psi$
from (1): 
\[
\Psi^{*}g_{\text{Misner}}=\tilde{g}
\]
This result demonstrates that the composite map $\Psi=\pi\circ\psi$
is an isometric immersion, which completes the proof. 
\end{proof}
We now proceed with the proof of Theorem~\ref{thm:H-Global Misner}.
\begin{thm*}
Let $(\tilde{M},\tilde{g})$ be an $n$--dimensional transverse signature-type
changing manifold with a transverse radical and with fixed $\mathcal{H}$--global
neighborhood $U=\bigcup_{q\in\mathcal{H}}U(q)\subset\tilde{M}$. Then
there exists an $\mathcal{H}$--global isometric embedding $\pi\circ\psi:U\longrightarrow\mathcal{M}_{\mathrm{Misner}}$
of $(U,\tilde{g}|_{U})$ into Misner space $\mathcal{M}_{\mathrm{Misner}}$.
\end{thm*}
\begin{proof}
Throughout this proof, all maps, coordinates, and geometric objects
are considered as defined on the fixed $\mathcal{H}$--global neighborhood
$U$ as in the theorem statement and the standing assumption. In particular,
the embedding $\psi:U\to\mathbb{R}^{1,N-1}$ and all related constructions
are restricted to $U$.

Let $\psi:U\subseteq\tilde{M}\to\mathbb{R}^{1,N-1}$ be the embedding
constructed in Theorem~\ref{thm:H-Global Minkowski}, whose image
lies in the region 
\[
\mathcal{R}=\{(\tau,y^{1},\dots,y^{N-1})\in\mathbb{R}^{1,N-1}\;|\;y^{1}-\tau>0\},
\]
 and let $\Gamma\subset O(1,N-1)$ be the one--parameter subgroup
of hyperbolic rotations acting in the $(\tau,y^{1})$--plane. Denote
by 
\[
\pi:\mathcal{R}\longrightarrow\mathcal{R}/\Gamma\cong\mathcal{M}_{\mathrm{Misner}}
\]
 the quotient map onto Misner space. We prove that the composition
\[
\pi\circ\psi:\tilde{M}\longrightarrow\mathcal{M}_{\mathrm{Misner}}
\]
is an isometric embedding, i.e. $\pi\circ\psi$ is an injective isometric
immersion which is a homeomorphism onto its image. While the property
of being an isometric immersion was established in Proposition~\ref{prop:misner_immersion},
we present here a self-contained proof that simultaneously demonstrates
all properties, thus directly establishing that the map is an isometric
embedding. The proof concludes with an alternative demonstration of
the isometric immersion property.

~

\textbf{I. Injectivity of }$\pi\circ\psi$\textbf{:} Suppose by contradiction
that there exist two distinct points $p_{1},p_{2}\in\tilde{M}$ with
\[
\pi\circ\psi(p_{1})=\pi\circ\psi(p_{2}).
\]
Since $\psi$ is an embedding, $\psi(p_{1})\neq\psi(p_{2})$, and
therefore the equality of images under $\pi$ implies that $\psi(p_{1})$
and $\psi(p_{2})$ lie on the same $\Gamma$--orbit in $\mathcal{R}$.
Equivalently, there exists $\gamma\in\Gamma$ with 
\[
\psi(p_{2})=\gamma\big(\psi(p_{1})\big).
\]

Write $\psi(p)=(\tau(p),Y^{1}(p),Y^{2}(p),\dots,Y^{N-1}(p))$ for
$p\in\tilde{M}$. By construction the action of $\Gamma$ is a one-parameter
subgroup of hyperbolic rotations $\Gamma\subset O(1,N-1)$ acting
only in the $(\tau,Y^{1})$--coordinates and leaving each transverse
spatial coordinate $Y^{A}$, $A\ge2$, fixed. Hence from $\psi(p_{2})=\gamma(\psi(p_{1}))$
we deduce 
\[
Y^{A}(p_{2})=Y^{A}(p_{1})\qquad\text{for every }A\ge2.\tag{1}
\]
Next recall that the embedding $\psi$ was chosen so that the first
component depends only on the absolute time coordinate $t$ on $\tilde{M}$,
namely 
\[
\tau(p)=-\tfrac{2}{3}\big(1+t(p)\big)^{3/2},
\]
with $t=\mathfrak{h}(t,\mathbf{\hat{x}})$ strictly increasing along
the leaves of the absolute time foliation (Corollary~\ref{cor: Time Function Foliation}).
In particular the assignment $t\mapsto\tau$ is strictly monotone,
hence $\tau(p_{1})=\tau(p_{2})$ if and only if $t(p_{1})=t(p_{2})$.

\textbf{~}\\ We consider two possibilities:

~

\textbf{Case (a):} $\emph{ \ensuremath{t(p_{1})=t(p_{2})}.}$ 

Then $\tau(p_{1})=\tau(p_{2})$. Because $\gamma$ acts only in the
$(\tau,Y^{1})$--plane and leaves $Y^{A}$ ($A\ge2$) fixed, from
$\psi(p_{2})=\gamma(\psi(p_{1}))$ we obtain $Y^{1}(p_{2})=Y^{1}(p_{1})$
as well. Thus all coordinates of $\psi(p_{1})$ and $\psi(p_{2})$
agree, so $\psi(p_{1})=\psi(p_{2})$, contradicting injectivity of
$\psi$.

~

\textbf{Case (b): $\ensuremath{t(p_{1})\neq t(p_{2})}$.}

Then $\tau(p_{1})\neq\tau(p_{2})$. Since $\psi(\tilde{M})$ is transverse
to the $\Gamma$--orbits by Proposition~\ref{prop:embedding transverse to the orbital foliation},
the embedded leaves $\psi\big(\mathfrak{h}^{-1}(t)\big)$ (the images
of the $t$--level slices) are transverse to the one-dimensional
orbits. Transversality implies that any given $\Gamma$--orbit meets
each embedded $t$--level in at most one point: near any intersection
the orbit crosses the leaf (rather than being tangent to it), so two
distinct intersections with the same orbit would force a tangency
by connectedness and the intermediate value property (the continuous
map from the orbit to the $t$--coordinate would take two different
values and hence assume a critical value between them). Therefore
an orbit cannot meet the image of two distinct $t$--levels. This
contradicts the assumption of Case (b) that $\psi(p_{1})$ and $\psi(p_{2})$
(on the same orbit) lie on different $t$--levels. 

~

Hence neither case can occur, and we conclude $p_{1}=p_{2}$. Thus
$\iota=\pi\circ\psi$ is injective.

\textbf{~}

\textbf{II. Homeomorphism onto its Image: }The map $\pi\circ\psi:U\to(\pi\circ\psi)(U)$
is continuous as a composition of continuous maps. To prove the map
$\pi\circ\psi:U\subseteq\tilde{M}\to(\pi\circ\psi)(U)$ is a homeomorphism
onto its image, it suffices to show that it is an open map onto its
image (an injective continuous open map onto its image has continuous
inverse).

Let $\bar{V}\subset U\subseteq\tilde{M}$ be open in the subspace
topology. Since $\psi$ is an embedding defined on $U$, the set $\psi(\bar{V})$
is open in the subspace topology of $\psi(U)$. Therefore there exists
an open set $V\subset\mathcal{R}$ such that 
\[
\psi(\bar{V})=V\cap\psi(U).
\]
Apply the quotient map $\pi$ to obtain 
\[
(\pi\circ\psi)(\bar{V})=\pi\big(\psi(\bar{V})\big)=\pi\big(V\cap\psi(U)\big)=\pi(V)\cap\pi(\psi(U)).
\]
Since $\pi:\mathcal{R}\to\mathcal{R}/\Gamma$ is a quotient map arising
from a continuous group action, $\pi(V)$ is open in the quotient
$\mathcal{R}/\Gamma$. Hence $\pi(V)\cap\pi(\psi(U))$ is open in
the subspace topology on $\pi(\psi(U))$. Therefore $(\pi\circ\psi)(\bar{V})$
is open in $(\pi\circ\psi)(U)$, which shows that $\pi\circ\psi:U\to(\pi\circ\psi)(U)$
is an open map onto its image.

Thus $\pi\circ\psi$ is a continuous bijection from $U\subseteq\tilde{M}$
onto the subspace $(\pi\circ\psi)(U)$ whose inverse is continuous,
it follows that $\pi\circ\psi$ is a homeomorphism onto its image.

~

For completeness:

~

\textbf{III. Isometric immersion:} By construction, the embedding
$\psi:U\subseteq\tilde{M}\to\mathbb{R}^{1,N-1}$ is an isometric embedding.
This means that the metric induced by $\psi$ on $U$ is precisely
the original metric $\tilde{g}|_{U}$: 
\[
\psi^{*}\eta(X,Y)=\tilde{g}|_{U}(X,Y)
\]
for all tangent vectors $X,Y\in T_{p}U\subseteq T_{p}\tilde{M}$,
$p\in U$.

~

The map $\pi:\mathcal{R}\to M_{\text{Misner}}$ is the quotient map
to Misner space, which identifies points along the orbits of the group
action $\Gamma$. The induced metric on the quotient manifold $M_{\text{Misner}}$
is given by the Minkowski metric $\eta$ restricted to the subspace
of tangent vectors in $\mathcal{R}$ that are orthogonal to the orbits
of $\Gamma$. The projection map $\pi$ is an isometry on this subspace.

~

A key assumption of our construction is that the foliation of $U\subseteq\tilde{M}$
by the level sets of the absolute time function is transverse to the
orbits of the group action $\Gamma$ in Minkowski space. This transversality
condition ensures that the pushforward of any non-zero tangent vector
$X\in T_{p}\tilde{M}$, with $p\in U$, by the map $\psi$ results
in a vector $D\psi(X)$ that is never tangent to an orbit of $\Gamma$.
Consequently, $D\psi(X)$ is always in the subspace on which the projection
$D\pi$ acts as an isometry. Given this, we can now compute the induced
metric of the composite map $\pi\circ\psi$: 

~

$(\pi\circ\psi)^{*}\eta(X,Y)$ 

$=\eta(D(\pi\circ\psi)(X),D(\pi\circ\psi)(Y))$ 

$=\eta(D\pi(D\psi(X)),D\pi(D\psi(Y)))$.

\textbf{~}\\ Because the tangent vectors $D\psi(X)$ and $D\psi(Y)$
are in the subspace where $D\pi$ is an isometry (due to transversality),
the projection preserves their inner product: 

$\eta(D\pi(D\psi(X)),D\pi(D\psi(Y)))$ 

$=\eta(D\psi(X),D\psi(Y))$ 

= $\psi^{*}\eta(X,Y)$

$=\tilde{g}(X,Y)$.

\textbf{~}\\ Thus, we have shown that $(\pi\circ\psi)^{*}\eta=\tilde{g}|_{U}$.
This proves that the composite map $\pi\circ\psi$ preserves the metric
and is an isometric embedding.

~

\textbf{IV. Conclusion:} Combining Steps 1--3 we find that $\pi\circ\psi$
is an injective isometric immersion which is a homeomorphism onto
its image, i.e. an isometric embedding of $(U,\tilde{g}|_{U})$ into
Misner space $\mathcal{M}_{\mathrm{Misner}}$, as required. We refer
the reader to the Appendix for the supplementary calculation in coordinates.
\end{proof}
\begin{rem}
The condition ``the embedded manifold intersects each orbit exactly
once'' is precisely the definition of its image, $\psi(U)$, being
a fundamental domain for the action of the boost group $\Gamma$ on
the region $\mathcal{R}$. 
\end{rem}

\begin{cor}
Under the hypotheses of Theorem~\ref{thm:H-Global Misner}, if the
extra monotonicity hypothesis of Corollary~\ref{cor:injectivity-by-lemma-orbit-tangency}
is satisfied, no orbit can meet $\psi(U)$ in two distinct points
with different $t$--values. Together with the uniqueness argument
for identical $t$ (Case (a) in the main proof) this implies injectivity
of $\pi\circ\psi$. 
\end{cor}

\section{Global Isometric Embedding\label{sec:global Isometric-embedding}}

Crucially, for a $n$--dimensional canonical model with signature-type
change, we are able to explicitly construct a fully global isometric
embedding into $(n+1)$--dimensional Minkowski space. This represents
a significantly stronger result than a $\mathcal{H}$--global embedding
for this specific case, as it is achieved without the need for local
constructions. This embedding into Minkowski space then immediately
yields a global isometric embedding into Misner space via the quotient
map. 

~

This construction provides a direct geometric framework for a braneworld
scenario: a $4$--dimensional spacetime with signature-type change
is realized as a brane whose intrinsic geometry is perfectly preserved
when embedded in a $5$--dimensional Minkowski bulk. The manner in
which the extra dimension is treated determines the physical interpretation
of this embedding.

~

We design a toy model on $\mathbb{R}^{n}$ with coordinates $(t,x^{1},\ldots,x^{n-1})$,
equipped with the transverse signature-type changing metric 
\[
\tilde{g}=-t(dt)^{2}+{\displaystyle \sum_{i=1}^{n-1}}(dx^{i})^{2}.
\]
This metric becomes degenerate at $t=0$, where the determinant vanishes.
The signature transitions from Riemannian $(0,n)$ for $t<0$ to Lorentzian
$(1,n-1)$ for $t>0$. The pair $(\mathbb{R}^{n},\tilde{g})$ is a
singular semi-Riemannian manifold with the locus of signature change
at the hypersurface $\mathcal{H}:=\{(t,\mathbf{\hat{x}})\in\mathbb{R}^{n}:t=0\}$.

~

Now consider the canonical embedding $\imath\colon\mathbb{R}^{n-1}\longrightarrow\mathbb{R}^{n}$
given by 
\[
\imath(\mathbf{\hat{x}})\coloneqq(0,x^{1},\ldots,x^{n-1}).
\]
The metric induced on the hypersurface of signature change is the
pullback $\imath^{*}\tilde{g}$. Since $dt=0$ on the image of $\imath$,
the induced metric takes the form:
\[
\imath^{*}\tilde{g}={\displaystyle \sum_{i=1}^{n-1}}(dx^{i})^{2}.
\]
Since the induced metric is positive definite, the induced geometry
is Riemannian. The normal to the hypersurface $\mathcal{H}$ is spanned
by the vector field $\frac{\partial}{\partial t}$. At the hypersurface
$t=0$, we find that $\tilde{g}(\frac{\partial}{\partial t},\frac{\partial}{\partial t})=-t=0$,
which means the normal is a null vector field. Consequently, the radical
$Rad_{q}:=\{w\in T_{q}M:\tilde{g}(w,\centerdot)=0\}$ is spanned by
the normal vector field, $\textrm{span}(\{\frac{\partial}{\partial t}\})$,
and is therefore transverse to the hypersurface at any $q\in\mathcal{H}$.
We utilize this well-defined setup as our toy model universe.

~

\subsection{Global Isometric Embedding into Minkowski Space\label{subsec:Global-isometric-embedding Minkowski}}

To formulate the question of isometric embedding more precisely, we
ask whether there exists a mapping $f:\mathbb{R}^{n}\longrightarrow\mathbb{R}^{1,n}$
such that the Minkowski metric $\eta$ in $\mathbb{R}^{1,n}$ pulls
back to the given metric $\tilde{g}$ on the submanifold $f(\mathbb{R}^{n})$.
The answer to this question leads to the following proposition.
\begin{prop}
Let $(\mathbb{R}^{n},\tilde{g})$ be the $n$--dimensional signature-type
changing manifold with the metric $\tilde{g}=-t(dt)^{2}+\sum_{i=1}^{n-1}(dx^{i})^{2}$.
There exists a global isometric embedding $f:\mathbb{R}^{n}\longrightarrow\mathbb{R}^{1,n}$
into Minkowski space $(\mathbb{R}^{1,n},\eta)$.
\end{prop}

\begin{proof}
Our goal is to find a global mapping $f:\mathbb{R}^{n}\longrightarrow\mathbb{R}^{1,n}$
that constitutes an isometric embedding of $(\mathbb{R}^{n},\tilde{g})$
into the flat target space $(\mathbb{R}^{1,n},\eta)$. This means
we must find a set of coordinate functions for the embedding, say
\[
f(t,\hat{x})=(\vartheta(t,\hat{x}),\xi(t,\hat{x}),Y^{2}(t,\hat{x}),\ldots,Y^{n-1}(t,\hat{x})),
\]
such that the pullback of the Minkowski metric $\eta$ matches the
source metric $\tilde{g}$, i.e. $f^{*}\eta=\tilde{g}$. The metrics
are given by: 

\[
f^{*}\eta=-(d\vartheta)^{2}+(d\xi)^{2}+\sum_{i=2}^{n-1}(dY^{i})^{2},
\]

\[
\tilde{g}=-t(dt)^{2}+\sum_{j=1}^{n-1}(dx^{j})^{2}.
\]
We can simplify the problem by noting that the metrics are of a product
form (and can be separated into two parts). Let's look for an embedding
of the form: 
\[
\vartheta=\vartheta(t,x^{1}),\quad\xi=\xi(t,x^{1}),\quad Y^{i}=Y^{i}(t,x^{1})\quad\text{for }i=2,\ldots,n-1.
\]
The conditions on the embedding functions for the $Y^{i}$ coordinates
are trivially satisfied. We can therefore focus on finding an embedding
for the $2$--dimensional submanifold with metric $^{(2)}\tilde{g}=-t(dt)^{2}+(dx^{1})^{2}$
into a $3$--dimensional flat space with metric $^{(3)}\eta=-(d\vartheta)^{2}+(d\xi)^{2}+(dZ)^{2}$.
We can then extend this embedding trivially to the higher-dimensional
spaces. Setting $x^{1}=x$, this implies

~

\[
-t=-(\frac{\partial\vartheta(t,x)}{\partial t})^{2}+(\frac{\partial\xi(t,x)}{\partial t})^{2}+(\frac{\partial x(t,x)}{\partial t})^{2},
\]

\[
1=-(\frac{\partial\vartheta(t,x)}{\partial x})^{2}+(\frac{\partial\xi(t,x)}{\partial x})^{2}+(\frac{\partial x(t,x)}{\partial x})^{2},
\]

\[
0=-(\frac{\partial\vartheta(t,x)}{\partial t}\frac{\partial\vartheta(t,x)}{\partial x})+(\frac{\partial\xi(t,x)}{\partial t}\frac{\partial\xi(t,x)}{\partial x})+(\frac{\partial x(t,x)}{\partial t}\frac{\partial x(t,x)}{\partial x}).
\]
\textbf{~}\\ One can imagine the $x$--axis poking out of the $\vartheta$--$\xi$
plane, providing a tangible visualization of the geometric relationships.
Thus, without loss of generality, we only need to consider

\begin{equation}
-t=-(\frac{d\vartheta(t,x_{0})}{dt})^{2}+(\frac{d\xi(t,x_{0})}{dt})^{2},\label{eq:isometric embedding}
\end{equation}
where $x$ is fixed, reducing the embedded $2$--dimensional model
to a curve parametrized by $t$. It is reasonable to choose the initial
values of $\vartheta=\xi=0$ for $t=0$, such that we have $(0\leq\frac{d\vartheta}{dt})\wedge(0\leq\frac{d\xi}{dt})\ \forall t$.
The first requirement ensures, nota bene, that the hypersurface of
signature change goes through the origin. 

\textbf{~}\\ A promising ansatz for solving this underdetermined
system of equations leverages the geometric properties of Minkowski
space. We note that the locus of signature change at $t=0$ must correspond
to a null vector, which we assume is located at the origin of the
target space, i.e., $\vartheta(t=0)=\xi(t=0)=0$, i.e. $\vartheta=\xi=0$.
Furthermore, the segment of the embedded curve for $t<0$ must be
timelike, meaning it lies inside the light cone. Conversely, the segment
for $t>0$ must be spacelike, lying outside the light cone. These
geometric constraints naturally suggest an ansatz based on a hyperbola
$\vartheta^{2}-\xi^{2}=1$, which lies outside the light cone. We
then rotate it clockwise by $45$ degrees and shift it in such a way
that it passes through the origin (Figure~\ref{fig:hyperbola-embedding idea}).
This procedure yields 

\textbf{~}\\ $((\xi+(1/\text{\textsurd}2))\cos(\pi/4)-(\vartheta-(1/\text{\textsurd}2))\sin(\pi/4))^{2}-((\vartheta-(1/\text{\textsurd}2))\sin(\pi/4)+(\xi+(1/\text{\textsurd}2))\cos(\pi/4))^{2}$

~

$=\frac{1}{2}(2\xi+\sqrt{2})(\sqrt{2}-2\vartheta)=1$

\ 

$\Longleftrightarrow\xi=\frac{\sqrt{2}\vartheta}{\sqrt{2}-2\vartheta}$

~

$\Longleftrightarrow\vartheta=\frac{\sqrt{2}\xi}{2\xi+\sqrt{2}}$.

\textbf{~}\\ Plugging $\frac{d\xi}{d\vartheta}=\frac{2}{(\sqrt{2}-2\vartheta)^{2}}$
into

\[
-t=-(\frac{d\vartheta(t,x_{0})}{dt})^{2}+(\frac{d\xi(t,x_{0})}{dt})^{2}=\left(-1+(\frac{d\xi(t,x_{0})}{d\vartheta(t,x_{0})})^{2}\right)\left(\frac{d\vartheta(t,x_{0})}{dt}\right)^{2}
\]
 gives 

\[
\left(\frac{4}{(\sqrt{2}-2\vartheta)^{4}}-1\right)(\frac{d\vartheta}{dt})^{2}=-t\Leftrightarrow\left(\frac{4}{(\sqrt{2}-2\vartheta)^{4}}-1\right)(d\vartheta)^{2}=-t(dt)^{2}.
\]

~

\begin{figure}[H]
\centering{}\includegraphics[scale=1.11]{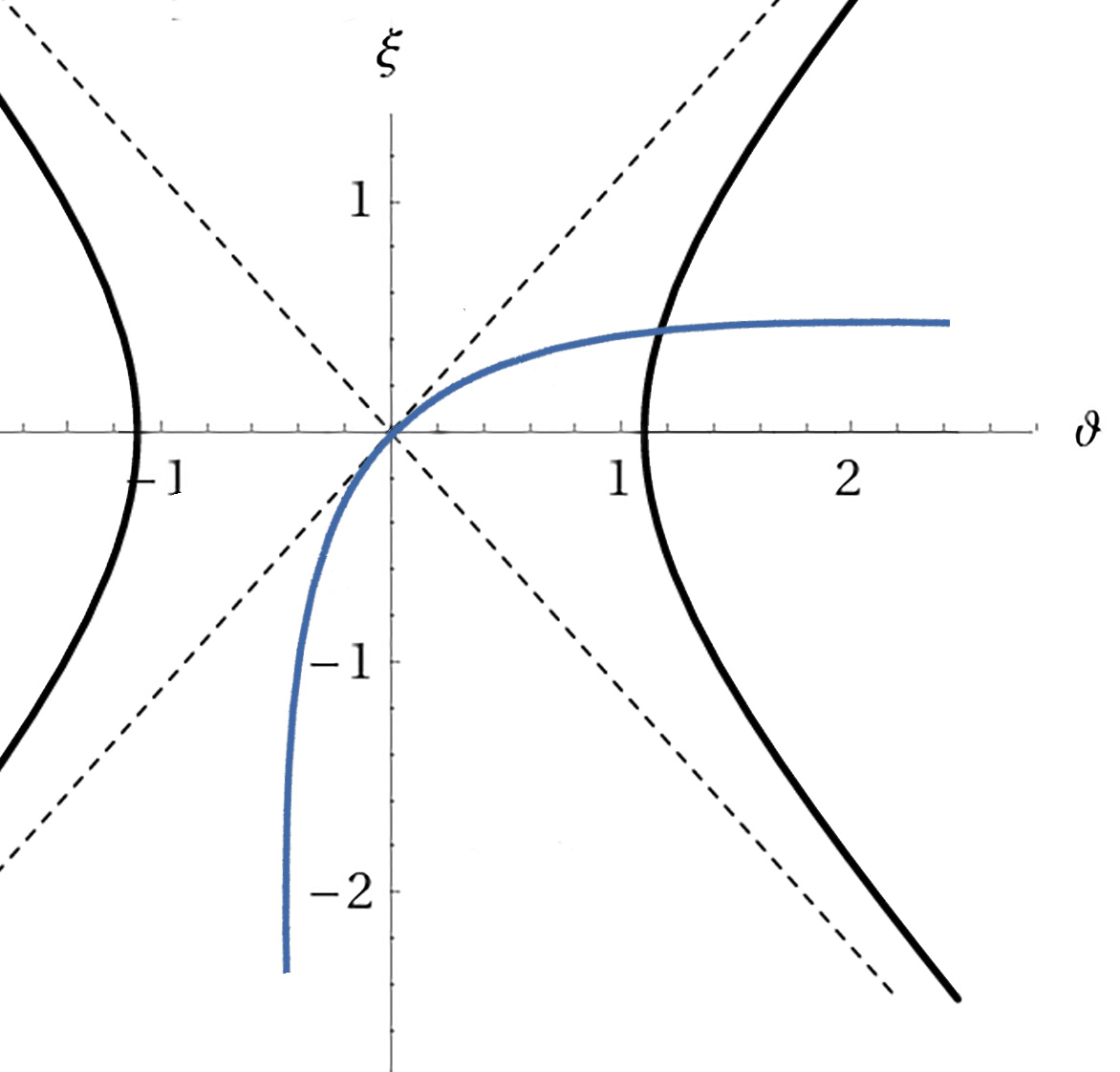}\caption{{\small{}\label{fig:hyperbola-embedding idea}The hyperbola $\frac{1}{2}(2\xi+\sqrt{2})(\sqrt{2}-2\vartheta)=1$
obtained from $\vartheta^{2}-\xi^{2}=1$ by rotating it by $45$ degrees
in clockwise direction and shift it such that it goes through the
origin. It should also be noted that the calculations are motivated
by a symmetry. The curve corresponding to the chosen relationship
between $\vartheta$ and $\xi$ is invariant under a reflection about
the line $\vartheta+\xi=0$. The possibility of freely specifying
a relationship between $\vartheta$ and $\xi$ arises from the underdetermination
of Equation~\ref{eq:isometric embedding}.}}
\end{figure}
\textbf{~}\\ Note that we have 

\textbf{~}

$\left(\frac{4}{(\sqrt{2}-2\vartheta)^{4}}-1\right)\geq0$ if $t\leq0$,
hence $4\geq(\sqrt{2}-2\vartheta)^{4}$,

$\left(\frac{4}{(\sqrt{2}-2\vartheta)^{4}}-1\right)\leq0$ if $t\geq0$,
hence $4\geq(\sqrt{2}-2\vartheta)^{4}$, 

$4=(\sqrt{2}-2\vartheta)^{4}$ if $t=0$, hence $(\vartheta=0)\vee(\vartheta=\sqrt{2})$. 

\textbf{~}\\ If we take into account the symmetry of the graph $\frac{4}{(\sqrt{2}-2\vartheta)^{4}}-1$
with respect to $\vartheta=\frac{1}{\sqrt{2}}$, we can consider the
absolute value of the equation

\begin{equation}
\left|\frac{4}{(\sqrt{2}-2\vartheta)^{4}}-1\right|(d\vartheta)^{2}=\left|t\right|(dt)^{2}.\label{eq:-1}
\end{equation}
Taking the square root of both sides of Equation~\ref{eq:-1} yields

~

\[
\sqrt{\left|\frac{4}{(\sqrt{2}-2\vartheta)^{4}}-1\right|}d\vartheta=\sqrt{\left|t\right|}dt
\]

\begin{equation}
\Longleftrightarrow\intop_{0}^{\vartheta}\sqrt{\left|\frac{4}{(\sqrt{2}-2\tilde{\vartheta})^{4}}-1\right|}d\tilde{\vartheta}=\intop_{0}^{t}\sqrt{\left|\tilde{t}\right|}d\tilde{t}=\frac{2}{3}\sqrt{\left|t\right|}^{3}\textrm{sgn}(t).\label{eq:explicit isometric embedding}
\end{equation}
This integral represents an exact, though implicit, solution for the. 
embedding functions. The equation establishes a relationship between
the embedding coordinate $\vartheta$ and the original coordinate
$t$ through a nonlinear transformation. This is a direct consequence
of the non-trivial nature of the source metric, which requires a non-trivial
map to the flat target space. To analyze the behavior of this solution,
we first note that 
\[
\underset{\vartheta\rightarrow\frac{1}{\sqrt{2}}}{\lim}\intop_{0}^{\vartheta}\sqrt{\left|\frac{4}{(\sqrt{2}-2\tilde{\vartheta})^{4}}-1\right|}d\tilde{\vartheta}=\infty.
\]
Then based on the series expansion at $\vartheta=0$, we have the
following approximations:

\textbf{~}\\ For $-1\ll\theta\ll\frac{1}{\sqrt{2}}$:

\[
\intop_{0}^{\vartheta}\sqrt{\left|\frac{4}{(\sqrt{2}-2\tilde{\vartheta})^{4}}-1\right|}d\tilde{\vartheta}\approx\intop_{0}^{\vartheta}\sqrt{\left|4\sqrt{2}\tilde{\vartheta}\right|}d\tilde{\vartheta}=2\sqrt[4]{2}\left(\intop_{0}^{\vartheta}\sqrt{\left|\tilde{\vartheta}\right|}d\tilde{\vartheta}\right)
\]

\[
=2\sqrt[4]{2}\left(\frac{2}{3}\sqrt{\left|\vartheta\right|}^{3}\textrm{sgn}(\vartheta)\right)=\frac{4\sqrt[4]{2}}{3}\sqrt{\left|\vartheta\right|}^{3}\textrm{sgn}(\vartheta)
\]

\[
\Longrightarrow\frac{4\sqrt[4]{2}}{3}\sqrt{\left|\vartheta\right|}^{3}\textrm{sgn}(\vartheta)\approx\frac{2}{3}\sqrt{\left|t\right|}^{3}\textrm{sgn}(t)
\]

\[
\Longleftrightarrow2\sqrt[4]{2}\sqrt{\left|\vartheta\right|}^{3}\textrm{sgn}(\vartheta)\approx\sqrt{\left|t\right|}^{3}\textrm{sgn}(t)
\]

\[
\Longrightarrow\vartheta\approx\frac{t}{2^{\frac{5}{6}}}.
\]
For $\vartheta\ll-1$
\[
\intop_{0}^{\vartheta}\underset{\approx1}{\underbrace{\sqrt{\left|\frac{4}{(\sqrt{2}-2\tilde{\vartheta})^{4}}-1\right|}}}d\tilde{\vartheta}\approx\vartheta.
\]
\end{proof}
~

The integral Equation~\ref{eq:explicit isometric embedding} is crucial
because it encodes the transformation from the source metric $^{(2)}\tilde{g}=-t(dt)^{2}+(dx)^{2}$
to the flat metric $^{(3)}\eta=-(d\vartheta)^{2}+(d\xi)^{2}+(dZ)^{2}$.
It represents the relationship between the original coordinates $(t,x)$
and the new embedding coordinates $(\vartheta,\xi,Z)$ that defines
the isometric embedding.

\begin{figure}[H]
\centering{}\includegraphics[scale=1.11]{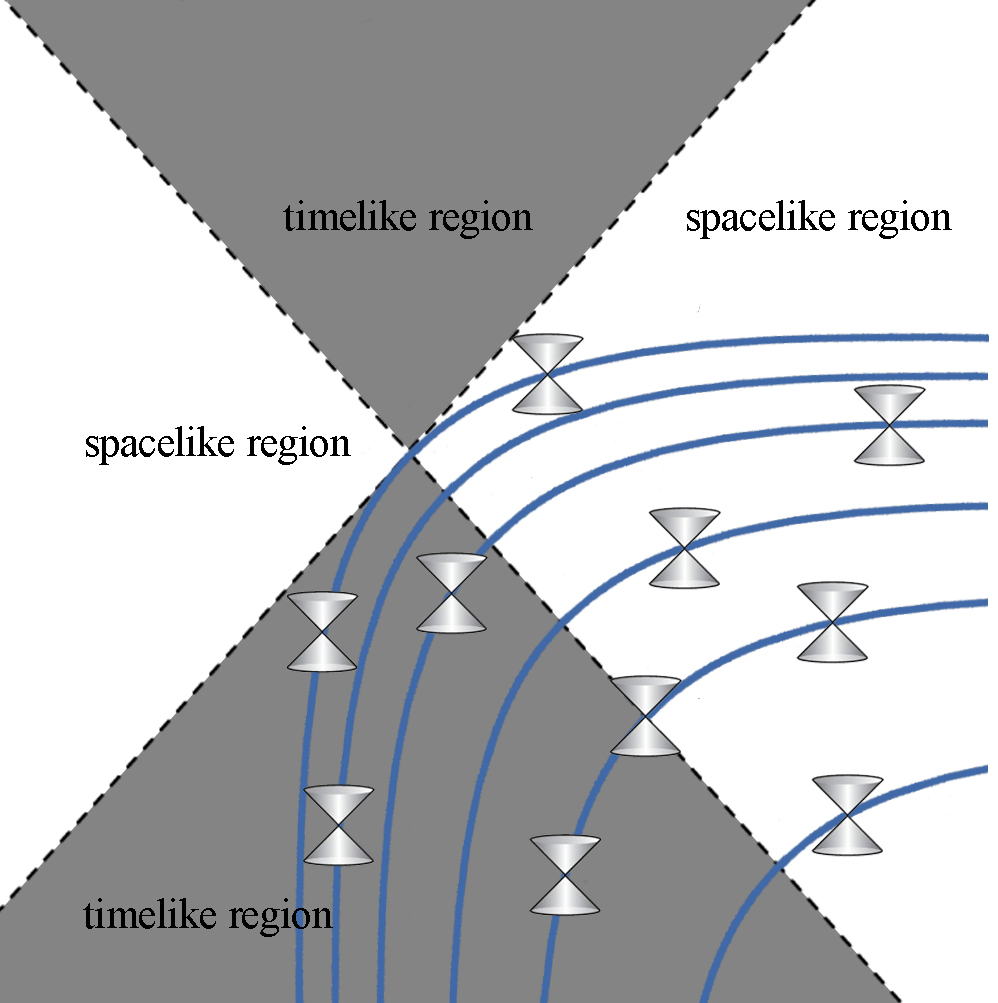}\caption{{\small{}\label{fig:embedding-family}A selection of curves from the
infinite family of isometric embeddings obtained by translating the
rotated and shifted hyperbola $\frac{1}{2}(2\xi+\sqrt{2})(\sqrt{2}-2\vartheta)=1$
along the line $\vartheta+\xi=0$, without further rotation. Each
curve represents a valid solution to the embedding condition described
in Equation}~\ref{eq:isometric embedding}.}
\end{figure}

It should be stressed that the hyperbola $\frac{1}{2}(2\xi+\sqrt{2})(\sqrt{2}-2\vartheta)=1$,
which is obtained from the equation $\vartheta^{2}-\xi^{2}=1$ by
a $45$-degree cockwise rotation and a specific translation, represents
just one member of an infinite family of solutions for this isometric
embeddingg (see Equation~\ref{eq:explicit isometric embedding}).
This derivation is motivated by symmetry: the curve defined by the
chosen relationship between $\vartheta$ and $\xi$ is invariant under
reflection about the line $\vartheta+\xi=0$. By translating it along
this line, without any additional rotation, we obtain all other members
of this solution family. A few such solutions are illustrated in Figure~\ref{fig:embedding-family}.

\subsection{Global Isometric Embedding into Misner Space}

The purpose of this section is to represent a signature-type changing
canonical model as a submanifold within Misner space via the process
of a global isometric embedding. Drawing on the results from Section~\ref{subsec:Global-isometric-embedding Minkowski},
we will demonstrate how the $n$--dimensional source manifold can
be isometrically embedded into the $(n+1)$--dimensional Misner space,
as stated in Proposition~\ref{prop:Toy Model Misner}:
\begin{prop*}
Let $(\mathbb{R}^{n},\tilde{g})$ be the $n$--dimensional signature-type
changing toy model manifold with the metric $\tilde{g}=-t(dt)^{2}+\sum_{i=1}^{n-1}(dx^{i})^{2}$.
Then there exists a global isometric embedding of the manifold $(\mathbb{R}^{n},\tilde{g})$
into $(n+1)$--dimensional Misner space $\mathcal{M}_{\text{Misner}}$.
\end{prop*}

\begin{proof}
The proof proceeds by demonstrating that the manifold $(\mathbb{R}^{n},\tilde{g})$
satisfies the conditions for the existence of a global isometric embedding
into Misner space, as established by our main theorems. The metric
$\tilde{g}$ is defined globally on all of $\mathbb{R}^{n}$ and is a specific instance of the canonical form described in Definition~\ref{def:Radical-Adapted-Gauss-Like-Coordinates},
with spatial components $g_{ij}=\delta_{ij}$.
The metric component $g_{tt}=-t=-\mathfrak{h}_{q}(t,\mathbf{\hat{x}})$ determines the boundary between the Lorentzian region $(t>0)$ and the Riemannian
region $(t<0)$ and corresponds to the absolue time function.  

The existence of a global isometric embedding into Minkowski space follows from Theorem~\ref{subsec:Global-isometric-embedding Minkowski}.
Since the  $\mathcal{H}$--global condition (Definition~\ref{def:H-Global})
holds on all of $\mathbb{R}^{n}$ (i.e., $U=\mathbb{R}^{n})$, Theorem~\ref{thm:H-Global Minkowski}, guarantees that an  $\mathcal{H}$--global isometric embedding $\psi$ into an ambient Minkowski space $(\mathbb{R}^{1,N_{h}},\eta)$ exists. In this case, the required Riemannian metric $h$ is given by $h=(dt)^{2}+\sum_{i=1}^{n-1}(dx^{i})^{2}$, 
so $(\mathbb{R}^{n},h)$ is flat Euclidean space. Consequently, the Nash embedding $\Phi : (\mathbb{R}^{n},h) \longrightarrow \mathbb{R}^{n}$ in Theorem~\ref{thm:H-Global Minkowski} is simply the identity map, 
$\Phi(t,\mathbf{\hat{x}})=(t,x^{1},\ldots,x^{n-1})$.


~

The image of the embedding, as given by Equation~\ref{eq:explicit isometric embedding}, is contained in the region $\mathcal{R}\subset\mathbb{R}^{1,n}$
where a Misner space can be constructed. By Theorem~\ref{thm:H-Global Misner},
which ensures the existence of an $\mathcal{H}$--global isometric
embedding into Misner space for such manifolds, and using  
Lemma~\ref{lem:orbit-tangency-H-global} and Corollary~\ref{cor:injectivity-by-lemma-orbit-tangency}, the composition of $\psi$ with the quotient map $\pi$ to Misner space, namely $\pi\circ\psi$,
defines a global isometric embedding of $(\mathbb{R}^{n},\tilde{g})$ into $\mathcal{M}_{\text{Misner}}$.
Since all conditions of the relevant theorems hold globally on $\mathbb{R}^{n}$,
the embedding is indeed global. For the full details of the
transversality argument, which are supplementary to this proof, see
the Appendix.
\end{proof}

\begin{figure}[H]
\centering{}\includegraphics[scale=1.11]{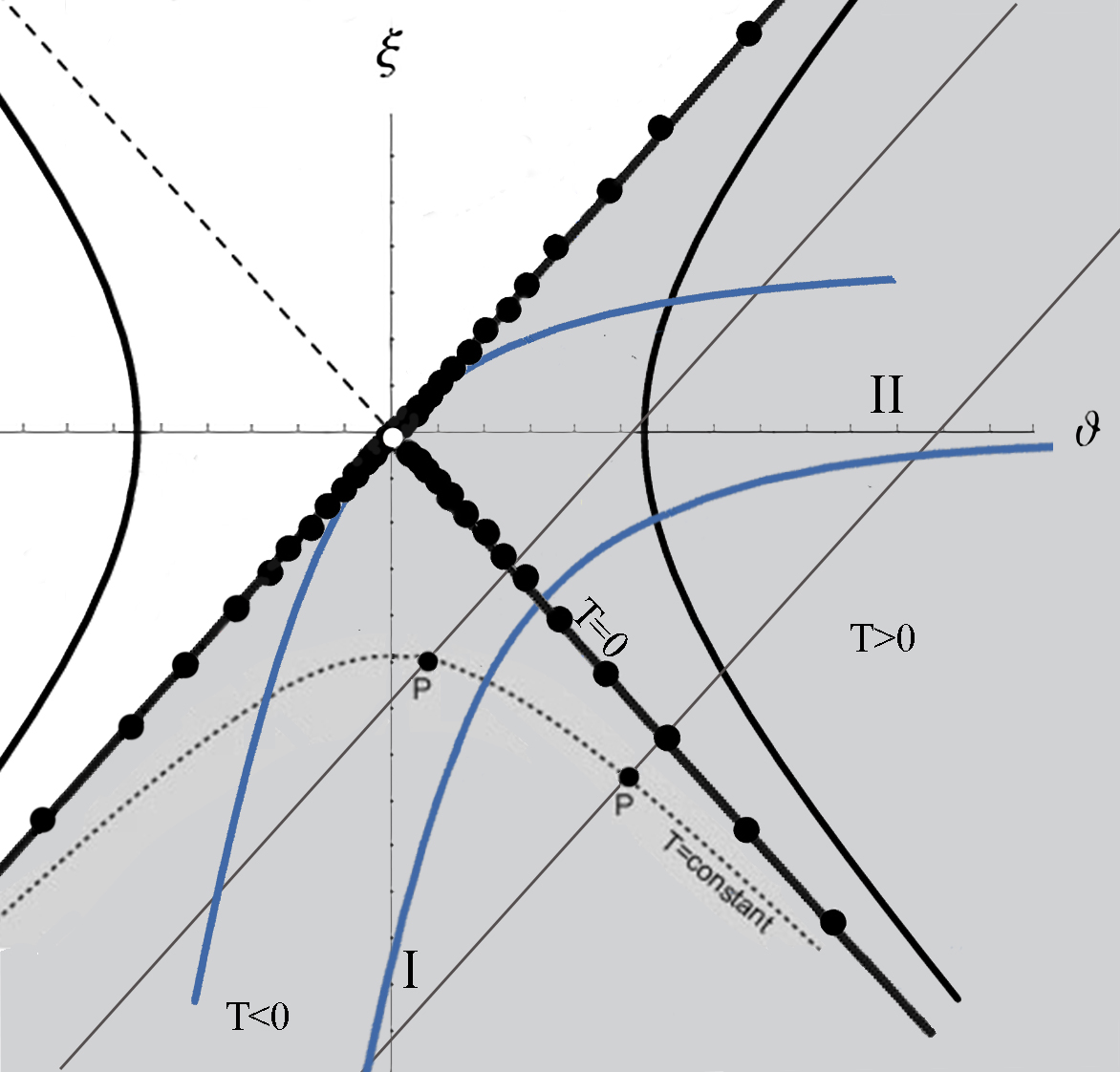}\caption{{\small{}\label{fig:Misner-Embedding}The $2$--dimensional $\vartheta$--$\xi$
plane, with the protruding $x^{i}$--axes suppressed, provides a
tangible visualization of the geometric relationships in two dimensions.
One hyperbola, given by $\frac{1}{2}(2\xi+\sqrt{2})(\sqrt{2}-2\vartheta)=1$
passes through the origin, while another hyperbola from the same family---translated
along the line $\vartheta+\xi=0$, represent the embedded $(\mathbb{R}^{n},\tilde{g})$
within Minkowski space. The shaded region $\mathcal{R}=I+II$ serves
as the covering space for Misner space, containing a representative
event $p$ that has an infinite number of copies. The high-order copies
of the event lying along the line $\vartheta+\xi=0$ (corresponding
to $T=0$) asymptotically approach the left chronology horizon.}}
\end{figure}

\begin{rem}
One should observe that in the covering space of Misner space, the
images of events under the group action lying along the null lines
$-\vartheta+\xi=0$ and $\vartheta+\xi=0$ asymptotically approach
the right and left chronology horizons, respectively, see Figure~\ref{fig:Misner-Embedding}.
These lines correspond to the future and past horizons. Furthermore,
there is only a single copy of the origin in the covering space, which
exists at the intersection of these lines but is excluded from the
Misner manifold itself. 

As a result, we must restrict our family of solutions for the isometric
embedding to those whose image does not pass through the origin, i.e.,
we restrict ourselves to those passing through the region $\{\vartheta+\xi=0,\xi<0\}$.
Specifically, this requires choosing a representative of the family
of solutions for the isometric embedding function $f$ (as defined
in Section~\ref{subsec:Global-isometric-embedding Minkowski}), such
that the image of the embedding lies entirely within the region defined
by $\mathcal{R}$. We will use the function $f$ to denote a representative
of this family of solutions that satisfies this constraint.

Under this restriction, the composition $\pi\circ f$ is a well-defined
map. It preserves the geometry of $\mathbb{R}^{n}$ because $f$ is
an isometry, while the projection $\pi$ acts by isometries within
$(\mathbb{R}^{1,n-1},\eta)$ to apply a hyperbolic rotation, transforming
the embedded spacetime coordinates accordingly.
\end{rem}

\section{Conclusion and Interpretation: A Braneworld and Kaluza-Klein Perspective}

The class of signature-type changing manifolds under consideration
includes physically significant scenarios, such as the Hartle-Hawking
\textquotedblleft no boundary\textquotedblright{} proposal, and offers
a coherent geometric approach to questions of signature change and
braneworld cosmologies. A key result of our analysis is the proof
of an isometric embedding of these manifolds into higher-dimensional
pseudo-Euclidean spaces, most notably Minkowski space. This embedding
reframes the investigation of signature-type change as a problem in
submanifold geometry, rendering it mathematically well-posed and free
of the usual pathologies.

~

The isometric embedding into Misner space after compactification is
particularly significant. It demonstrates how signature-type changing
spacetimes can be consistently realized within a higher-dimensional
Lorentzian setting. This provides an intriguing generalization of
the Kaluza-Klein paradigm. A central finding is that the $n$--dimensional
quotient manifold can exhibit signature-type change, even while the
higher-dimensional Misner metric remains Lorentzian throughout. This
novel situation can be described as \textquotedblleft signature change
without signature change\textquotedblright , as the signature change
is an effective property of the lower-dimensional manifold.

~

Our approach reframes the problem within the braneworld paradigm. Instead of treating a signature-changing universe as an isolated object, we regard it as a hypersurface, or brane, embedded in a higher-dimensional bulk spacetime. The key insight is that the geometry of our universe is induced by its embedding in this higher-dimensional space. From the intrinsic perspective, the geometry of a signature-changing manifold appears singular, whereas from the bulk perspective it may be entirely regular. This shift of viewpoint allows us to use the well-behaved geometry of the bulk to understand the more intricate, signature-changing geometry of the brane.

~

These results provide a new tool for theoretical model building in
both cosmology and higher-dimensional theories of gravity. By aligning
the mathematical idea of embedding spacetimes in higher dimensions
with constraints from brane scenarios, this work offers a deeper insight
into the nature of gravity by connecting the geometric structure of
spacetime with observable physical phenomena.

\begin{acknowledgement*}
This work was partially carried out while the author was a member
of Richard Schoen's research group at the University of California,
Irvine. The author gratefully acknowledges support from the Simons Center for Geometry and Physics, Stony Brook University, where part of this research was conducted, and partial support from the SNF Grant
No. 200021-227719. The author would also like to thank Leo Mathis, 
Iakovos Androulidakis, Federico Franceschini and Fabian Ziltener for insightful discussions
that helped shape the direction of this research. 
\end{acknowledgement*}

~

\appendix
\section*{Appendix}
\paragraph{Proof of Theorem~\ref{thm:H-Global Misner}: Calculation in Misner Coordinates}

~

\textbf{~}

In coordinate representation, the map $\pi$ on the $(\tau,y^{1})$--plane
of Minkowski space is defined by 
\[
\pi(\tau,y^{1})=\left(\frac{(y^{1})^{2}-\tau^{2}}{4},\,-2\log\left(\frac{y^{1}-\tau}{2}\right)\right)=(T,\phi),
\]
 see Equations~\ref{eq:Misner-Minkowski coordinates}. Then the overall
composition 
\[
\pi\circ\psi:\tilde{M}\to\mathcal{M}_{\mathrm{Misner}}
\]
 in $n$ dimensions is given by $\pi\circ\psi(t,\mathbf{\hat{x}})=$

\[
\left(\frac{\big(\Phi^{1}(t,\mathbf{\hat{x}})\big)^{2}-\left(\frac{2}{3}(1+t)\right)^{3}}{4},\;-2\log\left(\frac{\Phi^{1}(t,\mathbf{\hat{x}})+\left(\frac{2}{3}(1+t)\right)^{\frac{3}{2}}}{2}\right),\;\Phi^{2}(t,\mathbf{\hat{x}}),\ldots,\Phi^{N_{h}}(t,\mathbf{\hat{x}})\right).
\]
However, Misner space itself is defined as a quotient space under
the identification 
\[
(T,\phi,\ldots)\sim(T,\phi+2\pi k,\ldots),\quad k\in\mathbb{Z}.
\]
That is, points 
\[
\left(\frac{\big(\Phi^{1}(t,\mathbf{\hat{x}})\big)^{2}-\left(\frac{2}{3}(1+t)\right)^{3}}{4},\;-2\log\left(\frac{\Phi^{1}(t,\mathbf{\hat{x}})+\left(\frac{2}{3}(1+t)\right)^{\frac{3}{2}}}{2}\right),\Phi^{2}(t,\mathbf{\hat{x}}),\ldots\right)
\]
 and 
\[
\left(\frac{\big(\Phi^{1}(t,\mathbf{\hat{x}})\big)^{2}-\left(\frac{2}{3}(1+t)\right)^{3}}{4},\;-2\log\left(\frac{\Phi^{1}(t,\mathbf{\hat{x}})+\left(\frac{2}{3}(1+t)\right)^{\frac{3}{2}}}{2}\right)+2\pi k,\Phi^{2}(t,\mathbf{\hat{x}}),\ldots\right)
\]
 are identified as the same point for each integer $k$. The condition
that 
\[
T=\frac{\big(\Phi^{1}(t,\mathbf{\hat{x}})\big)^{2}-\left(\frac{2}{3}(1+t)\right)^{3}}{4}
\]
is strictly monotonic along the embedded manifold ensures that the
manifold does not \textquotedblleft fold back\textquotedblright {} {}in the time
dimension inside the target space. This monotonicity guarantees a
well-defined embedding in the temporal coordinate $T$.

~\paragraph{Proof of Proposition~\ref{prop:Toy Model Misner}: Supplementary Transversality Calculation}

~

\textbf{~}

This calculation is based on the \textquotedblleft Transversality
Proposition\textquotedblright ~\ref{prop:embedding transverse to the orbital foliation}.
In this model, the embedding's spatial components depend nontrivially
on the source coordinates, and we show that the transversality condition
leads to a contradiction, thereby proving transversality.
The embedding into $3$--dimensional Minkowski space is 
\[
\psi(t,x^{1})=\left(-\frac{2}{3}(1+t)^{\frac{3}{2}},\,t,\,x^{1}\right).
\]
 The coordinates of the target space are $(\tau,y^{1},y^{2})$, with
\[
\tau=-\frac{2}{3}(1+t)^{\frac{3}{2}},\quad y^{1}=t,\quad y^{2}=x^{1}.
\]
The tangent vectors of the embedded manifold are the partial derivatives:
\[
\frac{\partial\psi}{\partial t}=\left(-\sqrt{1+t},\,1,\,0\right),\quad\frac{\partial\psi}{\partial x^{1}}=(0,\,0,\,1).
\]
To check transversality, we consider the tangent vector of the orbit
\[
K=\tau\frac{\partial}{\partial y^{1}}+y^{1}\frac{\partial}{\partial\tau}.
\]
In components, 
\[
K=(y^{1},\tau,0).
\]
We ask if $K$ can be expressed as a linear combination of the tangent
vectors: 
\[
(y^{1},\tau,0)=a(-\sqrt{1+t},1,0)+b(0,0,1)=(-a\sqrt{1+t},a,b).
\]
Equating components yields: 
\[
\begin{cases}
y^{1}=-a\sqrt{1+t},\\
\tau=a,\\
0=b.
\end{cases}
\]
From the second equation, $a=\tau$. Substitute into the first: 
\[
y^{1}=-\tau\sqrt{1+t}.
\]
Using the embedding definitions $y^{1}=t$ and $\tau=-\frac{2}{3}(1+t)^{\frac{3}{2}}$,
we get 
\[
t=\frac{2}{3}(1+t)^{\frac{3}{2}}\sqrt{1+t}=\frac{2}{3}(1+t)^{2},
\]
which simplifies to 
\[
t=\frac{2}{3}(1+2t+t^{2}),
\]
or equivalently, 
\[
2t^{2}+t+2=0.
\]

The discriminant $\Delta=1-16=-15<0$ is negative, so no real solutions
for $t$ exist. Therefore, the assumption that $K$ lies in the tangent
space of the embedded manifold leads to a contradiction for all real
$t$. Hence, the embedding is transversal to the orbits.
\end{document}